\documentclass[11pt,a4paper]{article}

\usepackage{amsmath,amsfonts,amsthm,amssymb}
\usepackage[english]{babel}
\usepackage{graphicx}
\usepackage{array}
\usepackage{tikz}
\usetikzlibrary{shapes,snakes}

\newtheorem{theorem}{Theorem}
\newtheorem{lemma}[theorem]{Lemma}
\newtheorem{conjecture}[theorem]{Conjecture}
\newtheorem{corollary}[theorem]{Corollary}
\newtheorem{proposition}[theorem]{Proposition}

\numberwithin{theorem}{section}

\newcommand{\PP}{\mathbb{P}}
\newcommand{\EE}{\mathbb{E}}

\begin{document}

\title{Edge-partitioning a graph into paths: \\ beyond the Bar\'at-Thomassen conjecture\thanks{The first author was supported by ERC Advanced Grant GRACOL, project no. 320812. The second
author was supported by an FQRNT postdoctoral research grant and CIMI research fellowship. The fourth author was partially supported by the 
ANR Project STINT under Contract ANR-13-BS02-0007.}}

\author{Julien Bensmail$^a$, Ararat Harutyunyan$^b$, \\Tien-Nam Le$^c$, and St\'{e}phan Thomass\'{e}$^c$\\~\\
			\small $^a$Department of Applied Mathematics and Computer Science \\ \small Technical University of Denmark \\ \small DK-2800 Lyngby, Denmark\\~\\
			\small $^b$Institut de Math\'ematiques de Toulouse \\ \small Universit\'e Toulouse III \\ \small 31062 Toulouse Cedex 09, France\\~\\
			\small $^c$Laboratoire d'Informatique du Parall\'elisme \\ \small \'Ecole Normale Sup\'erieure de Lyon \\ \small 69364 Lyon Cedex 07, France}

\date{}

\maketitle

\begin{abstract}
In 2006, Bar\'at and Thomassen conjectured that there is a function $f$ such that, 
for every fixed tree $T$ with $t$ edges, every $f(t)$-edge-connected graph
with its number of edges divisible by $t$ has a partition of its edges into copies 
of $T$. 
This conjecture was recently verified by the current authors and Merker~\cite{BHLMT16+}.

We here further focus on the path case of the Bar\'at-Thomassen conjecture.
Before the aforementioned general proof was announced,
several successive steps towards the path case of the conjecture were made,
notably by Thomassen~\cite{Tha08,Thb08,Tha13},
until this particular case was totally solved by Botler, Mota, Oshiro and Wakabayashi~\cite{BMOW14}.
Our goal in this paper is to propose an alternative proof of the path case 
with a weaker hypothesis: Namely, we prove that there is a function $f$ such that every 
$24$-edge-connected graph with minimum degree $f(t)$ has an edge-partition into paths of length $t$
whenever $t$ divides the number of edges. We also show that~$24$ can be dropped to~$4$ when the
graph is eulerian.
\end{abstract}

\section{Introduction} \label{section:introduction}

Unless stated otherwise, graphs considered here are generally simple, loopless and undirected. Given 
a graph $G$, we denote by $V(G)$ and $E(G)$ its vertex and edge sets, respectively. Given a 
vertex $v$ of $G$, we denote by $d_G(v)$ (or simply $d(v)$ in case no ambiguity is possible) 
the degree of $v$ in $G$, i.e. the number of edges incident to $v$ in $G$. We denote by $\delta(G)$ 
and $\Delta(G)$ the minimum and maximum, respectively, degree of a vertex in $G$.
When $X$ is a subset of vertices of $G$, we denote by $d_X(v)$ the degree of 
$v$ in the subgraph of $G$ induced by $X\cup \{v\}$.
Given two graphs $G = (V, E)$ and $H = (V, F)$ with $F \subseteq E$, we denote by $G \backslash H$
the graph $(V, E \backslash F)$.

Let $G$ and $H$ be two graphs such that $|E(H)|$ divides $|E(G)|$. We say that $G$ is 
\textit{$H$-decomposable} if there exists a partition $E_1 \cup E_2 \cup ... \cup E_k$ of $E(G)$ 
such that every $E_i$ induces an isomorphic copy of $H$. We then call $E_1 \cup E_2 \cup ... \cup E_k$ 
an \textit{$H$-decomposition} of $G$.

\medskip

This paper is devoted to the following conjecture raised by Bar\'at and Thomassen in~\cite{BT06}, 
stating that highly edge-connected graphs can be decomposed into copies of any tree.

\begin{conjecture} \label{conjecture:barat-thomassen}
For any fixed tree $T$, there is an integer $c_T$ such that every 
$c_T$-edge-connected graph with its number of edges divisible by $|E(T)|$ 
can be $T$-decomposed.
\end{conjecture}

\noindent Conjecture~\ref{conjecture:barat-thomassen} was recently solved by the current authors and Merker in~\cite{BHLMT16+}.
For a summary of the progress towards the conjecture, we hence refer the interested reader to that paper.
Before this proof was announced, the path case of the conjecture had been tackled through successive steps.
First, the conjecture was verified for paths of small length, namely for $T$ being $P_3$ and $P_4$ by Thomassen~\cite{Tha08,Thb08},
where $P_\ell$ here and further denotes the path on $\ell$ edges.
Thomassen then proved, in~\cite{Tha13}, the conjecture for arbitrarily long paths of the form $P_{2^k}$.
Later on, Botler, Mota, Oshiro and Wakabayashi proved the conjecture for $P_5$~\cite{BMOW15+} before generalizing their arguments
and settling the conjecture for all paths~\cite{BMOW14}.

\medskip

Conjecture~\ref{conjecture:barat-thomassen} being now solved, 
many related lines of research sound quite appealing.
One could for example wonder, for any fixed tree~$T$,
about the least edge-connectivity guaranteeing the existence of $T$-decompositions.
We note that the proof of Conjecture~\ref{conjecture:barat-thomassen} from~\cite{BHLMT16+},
because essentially probabilistic, provides a huge bound on the required edge-connectivity,
which is clearly far from optimal.
Another interesting line of research,
is about the true importance of large edge-connectivity over large minimum degree in the statement of Conjecture~\ref{conjecture:barat-thomassen}.
Of course, one can notice that, to necessarily admit $T$-decompositions, graphs among some family must meet a least edge-connectivity condition.
We however believe that this condition can be lowered a lot, provided this is offset by a large minimum degree condition.
More precisely, we believe the following refinement of Conjecture~\ref{conjecture:barat-thomassen} makes sense.

\begin{conjecture}  \label{conjecture:min-degree-general}
There is a function $f$ such that, for any fixed tree $T$ with maximum degree $\Delta_T$, every 
$f(\Delta_T)$-edge-connected graph with its number of edges divisible by $|E(T)|$ and minimum degree at least
$f(|E(T)|)$ can be $T$-decomposed.
\end{conjecture}

In this paper, we make a first step towards Conjecture~\ref{conjecture:min-degree-general} by showing
it to hold when $\Delta_T \leq 2$, that is for the cases where $T$ is a path.

\begin{theorem} \label{theorem:24}
For every integer $\ell \geq 2$, there exists $d_\ell$ such that every $24$-edge-connected graph
$G$ with minimum degree at least~$d_\ell$ has a decomposition
into paths of length $\ell$ and an additional path of 
length at most $\ell$.
\end{theorem}

\noindent In particular, our proof of Theorem~\ref{theorem:24} yields a third proof of the path case of Conjecture~\ref{conjecture:barat-thomassen}.
It is also important mentioning that this proof is, in terms of approach, quite different from the one from~\cite{BMOW14}.

Let us, as well, again emphasize that the main point in the statement of Theorem~\ref{theorem:24} is that the required edge-connectivity, namely~$24$,
is constant and not dependent on the path length $\ell$. Concerning the optimal value as $f(2)$ mentioned in
Conjecture~\ref{conjecture:min-degree-general} (which is bounded above by~$24$, following Theorem~\ref{theorem:24}), 
a lower bound on it is~$3$ as there exist $2$-edge-connected graphs with arbitrarily large minimum
degree admitting no $P_\ell$-decomposition for some~$\ell$. To be convinced of this statement, just consider the following
construction. Start from the $2$-edge-connected graph $G$ depicted in~Figure~\ref{figure:example-3edgeco}, which admits
no $P_9$-decomposition. To now obtain a $2$-edge-connected graph with arbitrarily large minimum degree $d$ from it, just consider any $2$-edge-connected graph $H$ 
with sufficiently large minimum degree (i.e. at least~$d$) and verifying $|E(H)| \equiv 7 \pmod 9$. Then consider any 
vertex $v$ of $G$ with small degree, and add two edges from $v$ to a new copy of $H$. Repeating this transformation as long as 
necessary, we get a new graph which is still $2$-edge-connected, with minimum degree at least~$d$ and whose size is a multiple of~$9$
(due to the size of $G$ and $H$), but with no $P_9$-decomposition -- otherwise, it can be easily checked that $G$ would admit a $P_9$-decomposition,
a contradiction. 
% 
% Proving Theorem~\ref{theorem:24} with~$3$ instead of~$24$ can actually be done using a different proof scheme than the
% one used herein -- for that reason, the proof of this result, which uses some of the results and ideas from 
% the current paper, will appear in a later paper.

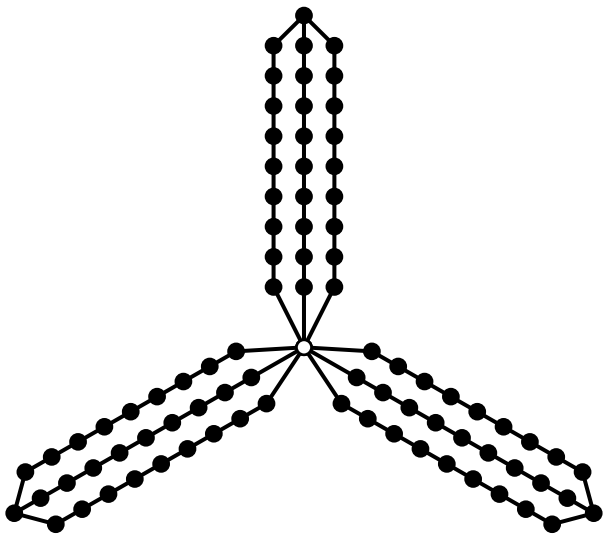
\begin{figure}[t]
\centering
\begin{tikzpicture}[inner sep=0.7mm,scale=0.4]
\node[draw, circle, line width=1pt, fill=white](u) at (0,0){};
\node[draw, circle, line width=1pt, fill=black] at (-1,2){};
\node[draw, circle, line width=1pt, fill=black] at (-1,3){};
\node[draw, circle, line width=1pt, fill=black] at (-1,4){};
\node[draw, circle, line width=1pt, fill=black] at (-1,5){};
\node[draw, circle, line width=1pt, fill=black] at (-1,6){};
\node[draw, circle, line width=1pt, fill=black] at (-1,7){};
\node[draw, circle, line width=1pt, fill=black] at (-1,8){};
\node[draw, circle, line width=1pt, fill=black] at (-1,9){};
\node[draw, circle, line width=1pt, fill=black] at (-1,10){};
\draw[-,line width=1.5pt] (-1,2) -- (-1,3) -- (-1,4) -- (-1,5) -- (-1,6) -- (-1,7) -- (-1,8) -- (-1, 9) -- (-1,10);
\node[draw, circle, line width=1pt, fill=black] at (0,2){};
\node[draw, circle, line width=1pt, fill=black] at (0,3){};
\node[draw, circle, line width=1pt, fill=black] at (0,4){};
\node[draw, circle, line width=1pt, fill=black] at (0,5){};
\node[draw, circle, line width=1pt, fill=black] at (0,6){};
\node[draw, circle, line width=1pt, fill=black] at (0,7){};
\node[draw, circle, line width=1pt, fill=black] at (0,8){};
\node[draw, circle, line width=1pt, fill=black] at (0,9){};
\node[draw, circle, line width=1pt, fill=black] at (0,10){};
\node[draw, circle, line width=1pt, fill=black] at (0,11){};
\draw[-,line width=1.5pt] (0,2) -- (0,3) -- (0,4) -- (0,5) -- (0,6) -- (0,7) -- (0,8) -- (0, 9) -- (0,10);
\node[draw, circle, line width=1pt, fill=black] at (1,2){};
\node[draw, circle, line width=1pt, fill=black] at (1,3){};
\node[draw, circle, line width=1pt, fill=black] at (1,4){};
\node[draw, circle, line width=1pt, fill=black] at (1,5){};
\node[draw, circle, line width=1pt, fill=black] at (1,6){};
\node[draw, circle, line width=1pt, fill=black] at (1,7){};
\node[draw, circle, line width=1pt, fill=black] at (1,8){};
\node[draw, circle, line width=1pt, fill=black] at (1,9){};
\node[draw, circle, line width=1pt, fill=black] at (1,10){};
\draw[-,line width=1.5pt] (1,2) -- (1,3) -- (1,4) -- (1,5) -- (1,6) -- (1,7) -- (1,8) -- (1, 9) -- (1,10);
\draw[-,line width=1.5pt] (-1,10) -- (0,11);
\draw[-,line width=1.5pt] (0,10) -- (0,11);
\draw[-,line width=1.5pt] (1,10) -- (0,11);
\draw[-,line width=1.5pt] (0,0) -- (-1,2);
\draw[-,line width=1.5pt] (0,0) -- (0,2);
\draw[-,line width=1.5pt] (0,0) -- (1,2);
%%%%%%%%%%%%%%%%%%
%%%%%%%%%%%%%%%%%%
%%%%%%%%%%%%%%%%%%
%%%%%%%%%%%%%%%%%%
\begin{scope}[rotate=120]
\node[draw, circle, line width=1pt, fill=black] at (-1,2){};
\node[draw, circle, line width=1pt, fill=black] at (-1,3){};
\node[draw, circle, line width=1pt, fill=black] at (-1,4){};
\node[draw, circle, line width=1pt, fill=black] at (-1,5){};
\node[draw, circle, line width=1pt, fill=black] at (-1,6){};
\node[draw, circle, line width=1pt, fill=black] at (-1,7){};
\node[draw, circle, line width=1pt, fill=black] at (-1,8){};
\node[draw, circle, line width=1pt, fill=black] at (-1,9){};
\node[draw, circle, line width=1pt, fill=black] at (-1,10){};
\draw[-,line width=1.5pt] (-1,2) -- (-1,3) -- (-1,4) -- (-1,5) -- (-1,6) -- (-1,7) -- (-1,8) -- (-1, 9) -- (-1,10);
\node[draw, circle, line width=1pt, fill=black] at (0,2){};
\node[draw, circle, line width=1pt, fill=black] at (0,3){};
\node[draw, circle, line width=1pt, fill=black] at (0,4){};
\node[draw, circle, line width=1pt, fill=black] at (0,5){};
\node[draw, circle, line width=1pt, fill=black] at (0,6){};
\node[draw, circle, line width=1pt, fill=black] at (0,7){};
\node[draw, circle, line width=1pt, fill=black] at (0,8){};
\node[draw, circle, line width=1pt, fill=black] at (0,9){};
\node[draw, circle, line width=1pt, fill=black] at (0,10){};
\node[draw, circle, line width=1pt, fill=black] at (0,11){};
\draw[-,line width=1.5pt] (0,2) -- (0,3) -- (0,4) -- (0,5) -- (0,6) -- (0,7) -- (0,8) -- (0, 9) -- (0,10);
\node[draw, circle, line width=1pt, fill=black] at (1,2){};
\node[draw, circle, line width=1pt, fill=black] at (1,3){};
\node[draw, circle, line width=1pt, fill=black] at (1,4){};
\node[draw, circle, line width=1pt, fill=black] at (1,5){};
\node[draw, circle, line width=1pt, fill=black] at (1,6){};
\node[draw, circle, line width=1pt, fill=black] at (1,7){};
\node[draw, circle, line width=1pt, fill=black] at (1,8){};
\node[draw, circle, line width=1pt, fill=black] at (1,9){};
\node[draw, circle, line width=1pt, fill=black] at (1,10){};
\draw[-,line width=1.5pt] (1,2) -- (1,3) -- (1,4) -- (1,5) -- (1,6) -- (1,7) -- (1,8) -- (1, 9) -- (1,10);
\draw[-,line width=1.5pt] (-1,10) -- (0,11);
\draw[-,line width=1.5pt] (0,10) -- (0,11);
\draw[-,line width=1.5pt] (1,10) -- (0,11);
\draw[-,line width=1.5pt] (0,0) -- (-1,2);
\draw[-,line width=1.5pt] (0,0) -- (0,2);
\draw[-,line width=1.5pt] (0,0) -- (1,2);
\end{scope}
%%%%%%%%%%%%%%%%%%
%%%%%%%%%%%%%%%%%%
%%%%%%%%%%%%%%%%%%
%%%%%%%%%%%%%%%%%%
\begin{scope}[rotate=-120]
\node[draw, circle, line width=1pt, fill=black] at (-1,2){};
\node[draw, circle, line width=1pt, fill=black] at (-1,3){};
\node[draw, circle, line width=1pt, fill=black] at (-1,4){};
\node[draw, circle, line width=1pt, fill=black] at (-1,5){};
\node[draw, circle, line width=1pt, fill=black] at (-1,6){};
\node[draw, circle, line width=1pt, fill=black] at (-1,7){};
\node[draw, circle, line width=1pt, fill=black] at (-1,8){};
\node[draw, circle, line width=1pt, fill=black] at (-1,9){};
\node[draw, circle, line width=1pt, fill=black] at (-1,10){};
\draw[-,line width=1.5pt] (-1,2) -- (-1,3) -- (-1,4) -- (-1,5) -- (-1,6) -- (-1,7) -- (-1,8) -- (-1, 9) -- (-1,10);
\node[draw, circle, line width=1pt, fill=black] at (0,2){};
\node[draw, circle, line width=1pt, fill=black] at (0,3){};
\node[draw, circle, line width=1pt, fill=black] at (0,4){};
\node[draw, circle, line width=1pt, fill=black] at (0,5){};
\node[draw, circle, line width=1pt, fill=black] at (0,6){};
\node[draw, circle, line width=1pt, fill=black] at (0,7){};
\node[draw, circle, line width=1pt, fill=black] at (0,8){};
\node[draw, circle, line width=1pt, fill=black] at (0,9){};
\node[draw, circle, line width=1pt, fill=black] at (0,10){};
\node[draw, circle, line width=1pt, fill=black] at (0,11){};
\draw[-,line width=1.5pt] (0,2) -- (0,3) -- (0,4) -- (0,5) -- (0,6) -- (0,7) -- (0,8) -- (0, 9) -- (0,10);
\node[draw, circle, line width=1pt, fill=black] at (1,2){};
\node[draw, circle, line width=1pt, fill=black] at (1,3){};
\node[draw, circle, line width=1pt, fill=black] at (1,4){};
\node[draw, circle, line width=1pt, fill=black] at (1,5){};
\node[draw, circle, line width=1pt, fill=black] at (1,6){};
\node[draw, circle, line width=1pt, fill=black] at (1,7){};
\node[draw, circle, line width=1pt, fill=black] at (1,8){};
\node[draw, circle, line width=1pt, fill=black] at (1,9){};
\node[draw, circle, line width=1pt, fill=black] at (1,10){};
\draw[-,line width=1.5pt] (1,2) -- (1,3) -- (1,4) -- (1,5) -- (1,6) -- (1,7) -- (1,8) -- (1, 9) -- (1,10);
\draw[-,line width=1.5pt] (-1,10) -- (0,11);
\draw[-,line width=1.5pt] (0,10) -- (0,11);
\draw[-,line width=1.5pt] (1,10) -- (0,11);
\draw[-,line width=1.5pt] (0,0) -- (-1,2);
\draw[-,line width=1.5pt] (0,0) -- (0,2);
\draw[-,line width=1.5pt] (0,0) -- (1,2);
\end{scope}

\node[draw, circle, line width=1pt, fill=white](u) at (0,0){};
\end{tikzpicture}
\caption{Part of the construction for obtaining $2$-edge-connected graphs with arbitrarily large minimum degree but no $P_\ell$-decomposition for some $\ell$.}
\label{figure:example-3edgeco}
\end{figure}

\medskip

Very roughly, the proof of Theorem~\ref{theorem:24}
goes as follows. When the graph $G$ has an eulerian tour $\mathcal{E}$, a natural strategy to obtain
a $P_\ell$-decomposition of $G$ is to cut $\mathcal{E}$ into consecutive $\ell$-paths. Of course we
may be unsuccessful in doing so since several consecutive edges of $\mathcal{E}$ may be \textit{conflicting},
that is have common vertices, hence inducing a cycle. Note however that if every edge of $\mathcal{E}$ (and hence
of $G$) is already a path of length at least~$\ell$, then,
cutting pieces along $\mathcal{E}$, only its consecutive paths can be conflicting -- hence bringing the notion
of conflict to a very local setting. Following this easy idea, the proof consists in expressing $G$ as
a \textit{$(\geq \ell)$-path-graph} (i.e. a system of edge-disjoint paths of length at least $\ell$
covering all edges) $H$ with low conflicts between its paths, then making $H$ eulerian somehow while keeping
low conflicts, and eventually deducing a conflictless eulerian tour that can eventually be safely cut into $\ell$-paths.

One side fact resulting from our proof scheme is that when $G$ is eulerian,
making $H$ eulerian requires less edge-connectivity. This remark, and additional
arguments, allow us to also prove the following result.

\begin{theorem} \label{theorem:4eulerian2}
For every integer $\ell \geq 2$, there exists $d_\ell$ such that every $4$-edge-connected eulerian graph
 with minimum degree at least~$d_\ell$ has a decomposition
into paths of length $\ell$ and an additional path of 
length at most $\ell$.
\end{theorem}
% 
% \noindent Once again, the required edge-connectivity in Theorem~\ref{theorem:4eulerian2} is not optimal
% as it can be lowered to~$2$ -- which is optimal because of eulerianity. This as well
% will be proved in a later paper as our arguments for proving this are totally different from the ones used herein.

\medskip

This paper is organized as follows. We start by introducing and recalling preliminary tools and results
in Section~\ref{section:fractions}. The notion of path-graphs and some properties of these objects are
then introduced in Section~\ref{section:path-graphs}. Particular path-graphs, which we call \textit{path-trees},
needed to repair eulerianity of path-graphs are then introduced and studied in Section~\ref{section:path-trees}. With
all notions and results in hands, we then prove Theorems~\ref{theorem:24} and~\ref{theorem:4eulerian2} in Section~\ref{section:main}. 

%%%%%%%%%%%%%%%%%%%%%%%%%%%%%%%%%%%%%%%%%%%%%%%%%%%%%%%%%%%%%%%%%%%%%
%%%%%%%%%%%%%%%%%%%%%%%%%%%%%%%%%%%%%%%%%%%%%%%%%%%%%%%%%%%%%%%%%%%%%
%%%%%%%%%%%%%%%%%%%%%%%%%%%%%%%%%%%%%%%%%%%%%%%%%%%%%%%%%%%%%%%%%%%%%
%%%%%%%%%%%%%%%%%%%%%%%%%%%%%%%%%%%%%%%%%%%%%%%%%%%%%%%%%%%%%%%%%%%%%
%%%%%%%%%%%%%%%%%%%%%%%%%%%%%%%%%%%%%%%%%%%%%%%%%%%%%%%%%%%%%%%%%%%%%
%%%%%%%%%%%%%%%%%%%%%%%%%%%%%%%%%%%%%%%%%%%%%%%%%%%%%%%%%%%%%%%%%%%%%
%%%%%%%%%%%%%%%%%%%%%%%%%%%%%%%%%%%%%%%%%%%%%%%%%%%%%%%%%%%%%%%%%%%%%
%%%%%%%%%%%%%%%%%%%%%%%%%%%%%%%%%%%%%%%%%%%%%%%%%%%%%%%%%%%%%%%%%%%%%
%%%%%%%%%%%%%%%%%%%%%%%%%%%%%%%%%%%%%%%%%%%%%%%%%%%%%%%%%%%%%%%%%%%%%
%%%%%%%%%%%%%%%%%%%%%%%%%%%%%%%%%%%%%%%%%%%%%%%%%%%%%%%%%%%%%%%%%%%%%
%%%%%%%%%%%%%%%%%%%%%%%%%%%%%%%%%%%%%%%%%%%%%%%%%%%%%%%%%%%%%%%%%%%%%
%%%%%%%%%%%%%%%%%%%%%%%%%%%%%%%%%%%%%%%%%%%%%%%%%%%%%%%%%%%%%%%%%%%%%
%%%%%%%%%%%%%%%%%%%%%%%%%%%%%%%%%%%%%%%%%%%%%%%%%%%%%%%%%%%%%%%%%%%%%
%%%%%%%%%%%%%%%%%%%%%%%%%%%%%%%%%%%%%%%%%%%%%%%%%%%%%%%%%%%%%%%%%%%%%
%%%%%%%%%%%%%%%%%%%%%%%%%%%%%%%%%%%%%%%%%%%%%%%%%%%%%%%%%%%%%%%%%%%%%
%%%%%%%%%%%%%%%%%%%%%%%%%%%%%%%%%%%%%%%%%%%%%%%%%%%%%%%%%%%%%%%%%%%%%

\section{Tools and preliminary results} \label{section:fractions}

Let $H=(V,F)$ be a spanning subgraph 
of a graph $G=(V,E)$. Let $\alpha$ be some real number in $[0,1]$.
We say that $H$ is {\em $\alpha $-sparse} in $G$ if $d_H(v)\leq \alpha d_G(v)$ 
for all vertices $v$ of $G$. Conversely,
we say that $H$ is {\em $\alpha $-dense} in $G$ if $d_H(v)\geq \alpha d_G(v)$ for all 
vertices $v$ of $G$. 
We will also heavily depend on subgraphs of $G$ which are both (roughly) $\alpha $-sparse and 
$\alpha $-dense. We say that $H$ is an \emph{$\alpha$-fraction} of $G$
if $\alpha d_G(v) - 10 \ell^{\ell}  \leq  d_H(v) \leq \alpha d_G(v) + 10 \ell^{\ell}$. 

%It may of course be the case that,
%because of divisibility issue, there does not exist any $\alpha$-fraction of $G$. For this reason, we will
%allow $\alpha$-fractions to approximate the desired degrees within a reasonable constant error range. 

Given an (improper) edge-coloring $\phi$ of some graph $G$ and a color~$i$, for every vertex $v$ of $G$ we denote by $d_i(v)$
the number of $i$-colored edges incident to $v$. We call $\phi$ \textit{nearly equitable} if, for every vertex $v$ and every 
pair of colors $i \neq j$, we have $|d_i(v)-d_j(v)|\le2$.
We can now recall a result of de Werra (cf. \cite{dewerra}, Theorem 8.7), and its corollary concerning $1/k$-fractions.

\begin{proposition}\label{werra}
Let $k \geq 1$. Every graph has a nearly equitable improper $k$-edge-coloring.
\end{proposition}

\begin{proposition} \label{halfgraph}
Let $k \geq 1$. Every graph $G=(V,E)$ has a subgraph $H=(V,F)$
 such that $|d_H(v)- d_G(v)/k|\le 2$
for every vertex $v$. 
\end{proposition}

We now recall two results on oriented graphs. The first of these is a result of Nash-Williams (see \cite{NW60})
implying that any graph with large edge-connectivity admits a balanced orientation with
large arc-connectivity. In the following, a digraph $D$ is \emph{$k$-arc-strong} if the removal of any set
of at most $k-1$ arcs leaves $D$ strongly-connected. 

\begin{proposition} \label{prop: arc-strong}
Every $2k$-edge-connected multigraph has
an orientation $D$ such that $D$ is $k$-arc-strong and such that $|d^-(v) - d^+(v)| \leq 1$ for every vertex $v$.
\end{proposition}

The second result we recall is due to Edmonds (see \cite{E73}) and expresses a condition for a digraph to admit
many arc-disjoint rooted arborescences. In the statement, an \emph{out-arborescence}
of a digraph $D$ refers to a rooted spanning tree $T$ of $D$ whose arcs are oriented
in such a way that the root has in-degree 0, and every other vertex has in-degree~1. 

\begin{proposition} \label{prop: disj-arbor}
A directed multigraph with a special vertex $z$
has $k$ arc-disjoint out-arborescences rooted at $z$ if and only if 
the number of arc-disjoint paths between $z$ and any vertex is at least $k$.
\end{proposition}

We end this section recalling probabilistic tools we will need in the next sections (refer e.g. to \cite{MR02} for more details).
The first of these is the well-known Local Lemma.

\begin{proposition}[Lov\'{a}sz Local Lemma] \label{prop: Lovasz}
Let $A_1,..., A_n$ be a finite set of events in some probability space $\Omega$,
with $\PP[A_i] \leq p$ for all $i$. Suppose that each $A_i$ is mutually independent of 
all but at most $d$ other events $A_j$. If $4pd < 1$, then $\Pr[\cap_{i=1}^{n} \overline{A_i}] > 0$.
\end{proposition}

We will also require the use of the following 
concentration inequality due to McDiarmid \cite{M02} (see also~\cite{MR02}). 
In what follows, a \emph{choice} is defined to be 
%either (a) the outcome of a trial or (b) 
the position that a particular element gets mapped to in a permutation.

\begin{proposition}[McDiarmid's Inequality (simplified version)]
Let $X$ be a non-negative random variable, not identically 0, which is determined by 
%$n$ independent trials $T_1,..., T_n$ and 
$m$ independent permutations $\Pi_1,..., \Pi_m$. If there exist $d, r >0$ such that 
\begin{itemize}
 %\item changing the outcome of any one trial can affect $X$ by at most $d$; 
 \item interchanging two elements in any one permutation can affect $X$ by at most $d$, and
 \item for any $s>0$, if $X \geq s$ then there is a set of at most $rs$ choices whose outcomes certify that $X\geq s$,
\end{itemize}
then for any $0 \leq \lambda \leq \EE[X]$,
$$ \PP\left[|X-\EE[X]| > \lambda + 60d \sqrt{r\EE[X]}\right] \leq 4 e^{-\tfrac{\lambda^2}{8d^2r\EE[X]}}.$$
\end{proposition}

\section{Path-graphs} \label{section:path-graphs}

Let $G=(V,E)$ be a graph. A {\em path-graph} $H$ \textit{on} $G$ is a couple $(V,{\cal P})$ 
where ${\cal P}$ is a set of edge-disjoint paths of $G$.
The graph $\underline{H}=(V,F)$, where $F$ contains the edges of paths in ${\cal P}$, is called the \textit{underlying graph} of $H$. 
If $F=E$, then $H$ is called a \textit{path-decomposition} of $G$.
Two edge-disjoint paths of $G$ sharing an end $v$ are said \textit{conflicting} if they also intersect in 
another vertex different from $v$.
Equivalently, we say that two paths of $H$ issued from a same vertex are conflicting
if the corresponding paths in $\underline{H}$ are conflicting.

We denote by $\tilde{H}$ the multigraph on 
vertex set $V$ and edge set the multiset containing a pair $uv$ for each path from $u$ to $v$ in $\cal P$ (if $\cal P$ contains several paths from $u$ to $v$, we add as many 
edges $uv$). We now transfer the usual definitions of graphs to path-graphs. The {\em degree} 
of a vertex $v$ in $H$, denoted $d_H(v)$, is the degree (with multiplicity) of $v$ in $\tilde{H}$.
We say that $H$ is {\em connected} if $\tilde{H}$ is connected, that $H$ is \textit{eulerian} if $\tilde{H}$ is eulerian, 
and that $H$ is a {\em path-tree} if $\tilde{H}$ is a tree (even if the paths of $\cal P$ pairwise intersect).
From a tour in $\tilde{H}$, we naturally get a corresponding \textit{tour} in $H$. 
Such a tour is said \textit{non-conflicting} if every two of its consecutive paths are non-conflicting.

We need also to speak of the length of the paths in $\cal P$. Let us say that $H$ is 
an {\em $\ell$-path-graph} if all paths in $\cal P$ have length $\ell$, a {\em $(\ge \ell)$-path-graph} if all 
paths in $\cal P$ have lengths at least $\ell$, an {\em $(\ell_1,\ell_2,...)$-path-graph} if all paths in 
$\cal P$ have lengths among $\{\ell_1, \ell_2, ...\}$, and an \emph{$[\ell, \ell + i]$}-path-graph if 
all paths in $\cal P$ have length in the interval $[\ell, \ell + i]$.

In general, the paths of a path-graph $H=(V,\cal P)$ can pairwise intersect, and we would hence like to measure
how much.
For every vertex $v$, let ${\cal P} _H(v)$ be the set of paths incident with $v$ in $H$. 
The {\em conflict ratio} of $v$ is $$\mbox {conf}(v):=\frac{\max_{w\ne v} \big| \{P\in {\cal P} _H(v) : w \in P\}\big|}{d_H(v)}.$$
Now, regarding $H$, we set ${\rm conf}_G(H):=\max_{v}\mbox {conf}(v)$.
When the graph $G$ is clear from the context, we will often omit the subscript in the notation. Clearly we always have ${\rm conf}(H) \leq 1$.

\medskip

With all the terminology above in hand, we can now prove (or recall) properties of path-graphs. We start
by recalling that, as desired, eulerian path-graphs with somewhat low conflicts have non-conflicting
eulerian tours. This matter was actually already considered by Jackson (cf. \cite{J93}, Theorem 6.3), whose
result implies it for path-graphs with small conflict ratio.

For a vertex $v$, let $E_v$ be the set of edges incident to $v$. A \emph{generalised transition system} $S$
for a graph $G$ is a set of functions $\{S_v\}_{v \in V(G)}$ such that $S_v: E_v \to 2^{E_v}$
and whenever $e_1 \in S_v(e_2)$, we have that $e_2 \in S_v(e_1)$. We say that an eulerian tour $\mathcal{E}$ is \emph{compatible}
with $S$ if for all $v \in V(G)$, whenever $e_1 \in S_v(e_2)$ it follows that 
$e_1$ and $e_2$ are not consecutive edges in $\mathcal{E}$.

\begin{theorem}[Jackson \cite{J93}] \label{thm: Jackson}
Let $S$ be a generalised transition system for an eulerian graph $G$. Suppose that for each vertex $v \in V(G)$
and $e \in E_v$, we have that $|S_v(e)| \leq \frac{1}{2}d(v) - 2$. Then $G$ 
has an eulerian tour compatible with $S$.
\end{theorem}

From Theorem \ref{thm: Jackson}, the following result is immediate.

\begin{theorem}\label{eulerian}
Every eulerian $[\ell, \ell +3]$-path-graph $H$ with ${\rm conf}(H) \leq 1/2(\ell + 10)$
has a non-conflicting eulerian tour.
\end{theorem}

\begin{proof}
Let $P \in {\cal P} _H(v)$. The number of paths of ${\cal P} _H(v)$ conflicting with $P$ is at most
$\frac{1}{2(\ell + 10)} (\ell + 3) d_H(v) \leq \frac{1}{2}d_H(v) - 2$. The result now follows from
Jackson's theorem.
\end{proof}

We now prove that every graph with large enough minimum degree can be expressed as a $(\geq \ell)$-path-graph meeting particular properties.

\begin{theorem}\label{dense}
Let $\ell$ be a positive integer, and $\varepsilon$ be an arbitrarily small positive real number. There exists $L$ such that if $G=(V,E)$ is a graph with minimum degree at least $L$, then 
there is an $\ell$-path-graph $H$ on $G$ with ${\rm conf}(H) \leq \varepsilon$, and $d_H(v)/d_G(v)\in \big[ \frac{1- \varepsilon}{\ell},\frac{1+ \varepsilon}{\ell}\big]$ and $d_{G\backslash \underline{H}}(v)\le \varepsilon d_H(v)$ for all vertices $v$.
\end{theorem}

\begin{proof}
Let $c:=[\sqrt{L}]$ and $b:=[c^{2/3}]$, and pick $L$ so that $b\gg \ell$. According to Proposition~\ref{werra},
we can nearly equitably color the edges of $G$ with $\ell$ colors.
For every color $i$, applying Proposition~\ref{prop: arc-strong} we can orient the $i$-colored edges 
so that the numbers of in-edges and out-edges of color $i$ incident to every vertex $v$ differ by at most 1. Let $E_i^-(v)$ and $E_i^+(v)$ be the sets of $i$-colored in-edges and out-edges, respectively, incident to $v$.
Then, for every color~$i \in \{1, ..., \ell-1\}$, we have $$\big||E_i^-(v)|-|E_{i+1}^+(v)|\big|\le 3.$$
For the sake of convenience, we would like to have that $|E_i^-(v)|=|E_{i+1}^+(v)|$ for all $i$ and $v$.
To this end, we add a dummy vertex $v_0$ to $G$.
Now, if $|E_i^-(v)|-|E_{i+1}^+(v)|=k>0$, then we add $k$ dummy edges of color $i+1$ from $v$ to $v_0$ to equalize $|E_i^-(v)|$ and  $|E_{i+1}^+(v)|$. Similarly, if 
$|E_{i+1}^+(v)|-|E_i^-(v)|=k>0$, then we add $k$ dummy edges of color $i$ from $v_0$ to $v$. 

Now, for every $v\in V(G)$ and colour $i \in \{1,..., \ell\}$, we choose 
$r_{v,i}\in \{0,\ldots ,c-2\}$ such that $E_i^{-}(v)\equiv r_{v,i}$ (mod $c-1$). 
Since the minimum degree in each colour in $G$ is greater than $c(c-2)$, we can partition every set $E_i^{-}(v)$ into subsets of size $c$ and $c-1$ so that precisely $r_{v,i}$ of them have size $c$.  As $E_{i+1}^{+}(v)=E_i^{-}(v)$, we can 
similarly partition every set $E_{i+1}^{+}(v)$ into subsets of size $c$ and $c-1$ so that precisely $r_{v,i}$ of them have size $c$.

We call these subsets of edges {\it $i$-half cones} and {\it $(i+1)$-half cones}, respectively.
%
%the number of $i$-half cones of size $c$ (resp. of size $c-1$) in the partition of $N_i(v)$ is equal to the number of $j$-blades of size $c$ (resp. of size $c-1$) in the partition of $N_j(v)$. 
%We can therefore partition the edges of $N_{S(t)}(v)$ into {\it fans} which are unions of blades of the same size such that precisely one $i$-blade appears in the fan for every $i \in S(t)$. 
%In other words, a fan $\varphi$ at a vertex $v$ (with relation to $t$) is a subset of $N_{S(t)}(v)$ of size $c|S(t)|$ or $(c-1)|S(t)|$ such that all colours in $S(t)$ appear $c$ times or $c-1$ times in $\varphi$. We also call $\varphi$ a \textit{$t$-fan} to indicate the colours appearing in $\varphi$.
Now, for each vertex $v$ and color $i$, $ 1 \leq i \leq \ell - 1$, 
we arbitrarily pair $i$-half cones of $E_i^-(v)$ with $(i+1)$-half cones $E_{i+1}^+(v)$ in a way
such that in each pair the size of the two half cones are equal. We call such a 
pair an \emph{$i$-cone at vertex $v$}. Thus, an $i$-cone $\varphi$ at some vertex $v$ consists of an 
$i$-half cone $\varphi^{-}$ and an $(i+1)$-half cone $\varphi^{+}$ with $|\varphi^{-}| = |\varphi^{+}|$.  
Note that an edge $e$ of color $i$ directed from 
a vertex $u$ to a vertex $v$ in $G$ appears both in an $i$-half cone of $E_i^{+}(u)$ as 
well as in an $i$-half cone of $E_i^{-}(v)$, but we do not require these two $i$-half cones 
to have the same size. By convention, we do not create a cone at the dummy vertex $v_0$. However,
each edge $uv_0$ will still be inside a cone at vertex $u$.
We also remark that the 1-half cones of $E_{1}^{+}(v)$ and 
and $\ell$-half cones of $E_{\ell}^{-}(v)$ do not get paired with other
half cones. Nevertheless, we will adopt the convention that whenever we talk of a general 
cone $\varphi$, we will assume that $\varphi$ might also 
consist of a single 1-half cone or $\ell$-half cone of the aforementioned type.

%Note that
%edges of $E_1^{+}(v)$ of $\ell$  are special in that they miss a 
%set of in-edges and out-edges, respectively.

\medskip
 
We now have a fixed set of cones on $G$. To obtain our desired path-graph,
we will use the cone structure to construct rainbow paths of length $\ell$, i.e., paths
where for all $i$ the $i^{th}$ edge
of every path is of color $i$. 
One way to obtain this is to randomly match edges of the two half cones of every cone. Indeed, this is what we do.
For each cone $\varphi$ we carry out random permutations $\pi_{\varphi}^{-}$ of the edges of $\varphi^{-}$ 
and $\pi_{\varphi}^{+}$ of the edges of $\varphi^{+}$. We then pair the edges $\pi_{\varphi}^{-}(k)$
and $\pi_{\varphi}^{+}(k)$
for each $1 \leq k \leq c$.
If $\varphi$ is actually a special 1-half cone or $\ell$-half cone, then
there is only one random permutation performed at $\varphi$, which
will have no affect on the decomposition as will be apparent shortly.
Note that each edge $e = uv$ of $G$, with the exception of some edges of 1-cones, 
$\ell$-cones and the dummy edges,
is in exactly two cones - one centered at $u$ and the second centered at $v$. 
Thus, $e$ is involved in two random permutations corresponding to the two
permutations of the two half cones containing it. Therefore, given the random matchings, each non-dummy edge $e = uv$ of color $i$, $ 1 < i < \ell$, is paired exactly with one edge of color
$i-1$ (which enters $u$) and one edge of color $i+1$ (which exits $v$). 
From an arbitrary edge, we can thus go forward and backward  by edges paired with it until we reach 
edges of color $\ell$ or $1$ (unless we reach dummy edges). Thus, the random matchings yield a natural
decomposition of all edges of $G$ into edge-disjoint \emph{walks}. 
Unfortunately, some of the walks will not be paths.
%   
% We say that the resulting random matchings obtained inside all cones 
% form a \textit{configuration} of $G$. 
% Given a configuration, all edges of color $i$ except dummy edges, 
% are paired with an edge of color $i-1$ if $i>1$ and an edge of color $i+1$ if $i<\ell$.
% From an arbitrary edge, we can go forward and backward  by edges paired with it until we reach 
% edges of color $\ell$ or $1$ or dummy edges. 
% When a path enters $v_0$ by a dummy edge, it will not
% % be extended any longer. 
% Hence, for each configuration, we achieve a decomposition of $G$ into 
% edge-disjoint \textit{walks} of length at most $\ell$. 
We will divide the walks into three types. Of the first type are those walks which are paths,
and thus by construction they are necessarily isomorphic to $P_{\ell}$.
A walk that is not a path and which does not use the dummy vertex $v_0$ 
is called a \textit{bad walk}; note that every bad walk is of length $\ell$. A walk that uses the 
dummy vertex $v_0$ is called a \textit{short walk}. Note that a short walk is no longer 
extended from $v_0$ as there is no cone centered at $v_0$.

For each cone $\varphi$, there are $c-1$ or $c$ walks via $\varphi$, depending on $|\varphi|$.
We will show that, with high probability, the number of bad or short walks via $\varphi$
are negligible compared to $c$. We then will argue that proving this statement for all the cones
is sufficient for us to extract a dense path-graph from $G$. 
%Let $\Omega$ be the space of all configurations. 
%Next, we will examine the probability of some events in $\Omega$. 

\medskip

Denote $P_{\ell}:= x_0x_1...x_{\ell}$.
We first focus on bad walks.
Suppose that $\varphi$ is a $k$-cone at some vertex $v$, and $i$, $j$ are two colors. 
%We can assume that $|k-i|\le|k-j|$, and $k<j$, otherwise, we do the same in reverse direction.
We say that a bad walk $P=u_0u_1...u_{\ell}$ going through $\varphi$ is 
\emph{$(i,j)$-bad} if its $i^{th}$ vertex and $j^{th}$ vertex are 
the same, that is, $u_j=u_i$. 
Let $A_{\varphi}(i,j)$ be the event that the number of $(i,j)$-bad walks going through the 
cone $\varphi$ is greater than $b$. We will show that
$\PP[A_{\varphi}(i,j)] < 4e^{-c^{2/3}/64}$. 

Denote by $P_{i,k}$ and $P_{j,k}$ to be the subpaths from $x_i$ to $x_k$, and $x_j$ to $x_k$ in
$P_{\ell}$, respectively. In case one of these paths is contained in another,
we may assume that $P_{i,k}$ is contained in $P_{j,k}$. Let $x_{j'}$ be the neighbor
of $x_j$ in $P_{j,k}$. Note that $j' \in \{j-1, j+1 \}$. 
%We will assume that $k < j$ -- an identical argument in the reverse
%direction will give the result for $k > j$.
Let $\mathcal{P}_\varphi$ be the set of walks that go through $\varphi$ 
which are not short. Clearly,  $|\mathcal{P}_\varphi|
\leq c$. 
%To upper bound $\Pr[A_{\varphi}(i,j)]$, we only need
%to know how the walks $\mathcal{P}_\varphi$ are distributed among the $(j-1)$-cones.
%It turns out that for our purposes the distribution of $\mathcal{P}_\varphi$ is irrelevant to obtain
%the desired bound.

%Note that $ j > i + 1$ since otherwise a walk cannot be $(i,j)$-bad. 
We define
$\Omega_{j'}$ to be the set of all $j'$-cones in $G$ if $j' = j-1$,
and the set of all $j$-cones if $j'=j+1$. Let $\Pi$ be an arbitrary
but fixed outcome of all permutations at all cones except the set of permutations 
on $\Omega_{j'}$.
In other words, given $\Pi$, we only need to know the outcomes of the
set of permutations $\{\pi_{\varphi}^{+}, \pi_{\varphi}^{-}
\mid \varphi \in \Omega_{j'} \}$ to know the decomposition of the walks in $G$.
We will condition on $\Pi$; that is, we will show that $\PP[A_{\varphi}(i,j)] \mid \Pi] < 
4e^{-c^{2/3}/64}$ for any $\Pi$. Clearly, since $\Pi$ is arbitrary, this is sufficient to give us the 
uniform bound $\PP[A_{\varphi}(i,j)] < 
4e^{-c^{2/3}/64}$.
 
Let $\mathcal{P}'_{\varphi}$ denote the set of walks $\mathcal{P}_{\varphi}$ conditional on $\Pi$. 
Let $X_\varphi$ be the number of $(i,j)$-bad walks going through the cone 
$\varphi$ conditional on $\Pi$.  
By fixing $\Pi$, the set $\mathcal{P}'_{\varphi}$ is also fixed. Indeed, 
each $P' \in \mathcal{P}'_{\varphi}$ is a partial subwalk, where we know
the vertex of $P'$ that lies in some half-cone of a cone $\psi \in 
\Omega_{j'}$. 
Note that the vertex $u_i$ of $P'$ corresponding to $x_i$ is already known. 
Moreover, the vertex $u_{j'}$ corresponding to the vertex $x_{j'}$ is known
as well. 

Note that whether $P'$ is $(i,j)$-bad depends only on the permutations
$\pi^{-}_{\psi}$ and $\pi^{+}_{\psi}$. Note that there are $c-1$ or $c$
different images possible to match $u_{j'}$ when the random permutations
$\pi^{-}_{\psi}$ and $\pi^{+}_{\psi}$ are carried out, and only one of which could possibly be $u_i$. Thus, the probability that $P'$ is $(i,j)$-bad
is at most $\tfrac{1}{c-1}$.

Now, by linearity of expectation,
$$\mathbb{E}[X_{\varphi}] \leq |\mathcal{P}_{\varphi}| \cdot \frac{1}{c-1} \leq \frac{c}{c-1}.$$
%Let $Y_{\varphi_v} := X_{\varphi} + c^{2/3}$. Note that $$c^{2/3} \leq \mathbb{E}[Y_{\varphi_v}] \leq c^{2/3} + \frac{c}{c-1}.$$
We will apply McDiarmid's inequality to the random variable $Y_{\varphi}$ defined by $ Y_{\varphi}:= X_{\varphi} + c^{2/3}$. Clearly $\mathbb{E}[Y_{\varphi}] = \mathbb{E}[X_{\varphi}] + c^{2/3} \in [ c^{2/3}, c^{2/3} + 2 ]$.
Only the permutations $\pi^{-}_{\psi}, \pi^{+}_{\psi}$ with $\psi \in 
\Omega_{j'}$ affect $X_{\varphi}$ and thus $Y_{\varphi}$. If two elements in one of these permutations are interchanged, then the structure of two walks 
in $\mathcal{P}_{\varphi}$ changes. However, clearly the number of $(i,j)$-bad walks in $\mathcal{P}_{\varphi}$ cannot change by more than 2. Thus, we can choose $d=2$ in McDiarmid's inequality.

If $Y_{\varphi} \geq s$, then $X_{\varphi} \geq s - c^{2/3}$, and thus at least $s - c^{2/3}$ of the walks in $\mathcal{P}_{\varphi}$ are $(i, j)$-bad. 
Let $P'\in \mathcal{P}_{\varphi}'$ be a subwalk of a walk $P$ that is counted by $X_{\varphi}$. As before, let $u_i = u_j$ denote the images of $x_i$ and $x_j$ in $P$, and $\psi \in \Omega_{j'}$ the cone through which $P'$ passes.
To verify that $P$ is $(i, j)$-bad, we only need to reveal the two elements
$\pi_{\psi}^{+}(s), \pi_{\psi}^{-}(s)$, where $ 1 \leq s \leq c$ is the value
such that the edge $u_{j'}u_j \in \{ \pi_{\psi}^{+}(s), \pi_{\psi}^{-}(s) \}$.

Thus, $X_{\varphi} \geq s - c^{2/3}$ can be certified by the outcomes of $2(s - c^{2/3}) < 2s$ choices and we can choose $r=2$ in McDiarmid's inequality.
By applying McDiarmid's inequality to $Y_{\varphi}$ with $\lambda =\EE[Y_{\varphi}]$, $d=2$, $r=2$, we get 
$$\PP\left[|Y_{\varphi}- \EE[Y_{\varphi}]| > \EE[Y_{\varphi}] + 120 \sqrt{2\EE[Y_{\varphi}]}~\right] \leq 4 e^{-\tfrac{\EE[Y_{\varphi}]}{64}} \leq 4e^{-\tfrac{c^{2/3}}{64}}$$
and thus
$\PP\left[X_{\varphi} > 2c^{2/3} \right] \leq 4e^{-c^{2/3}/64}$.
So we have $\PP[A_{\varphi}(i,j)|\Pi] < 4e^{-c^{2/3}/64}$. Since $\Pi$ is arbitrary 
it follows that $\PP[A_{\varphi}(i,j)] < 4e^{-c^{2/3}/64}$.
Let $A_{\varphi}$ be the event that there are more than $\ell^2b$ bad walks via $\varphi$. $$\PP[A_{\varphi}]\le \PP\Big[\bigcup_{\forall i,j}A_{\varphi}(i,j)\Big]\le \sum_{\forall i,j}\PP[A_{\varphi}(i,j)]< 4 \ell^2 e^{-c^{2/3}/64}.$$

We still consider
the same cone $\varphi$.
For an integer $j \neq k$ and 
vertex $u$, let $B_\varphi(j, u)$ be the event that the number of walks via $\varphi$ whose image of $x_j$ is $u$ is greater than $b$,
and let $B_\varphi(u)$ be the event that the number of walks of $\varphi$ containing $u$ is greater than $\ell b$. 

We show that $\PP[B_\varphi(j, u)] < 4e^{-c^{2/3}/64}$. As the computation
is virtually identical to the case of $\PP[A_{\varphi}(i,j)]$, we only
highlight the differences. 
As before, let $x_{j'}$ be the vertex adjacent to $x_j$ on $P_{j,k}$, and
let $\Pi$ be an arbitrary
but fixed outcome of all permutations at all cones except the set of permutations 
on $\Omega_{j'}$. It suffices to show that 
$\PP[B_\varphi(j, u) \mid \Pi] < 4e^{-c^{2/3}/64}$. 

Let $X_{\varphi}$ denote the random variable conditional on $\Pi$ which counts the number of walks in $\mathcal{P}_{\varphi}$ where $u$ is the image of $x_j$. 
The vertex $u$ appears at most once in each cone of $\Omega_{j'}$, so by linearity of expectation we have 
$$\mathbb{E}[X_{\varphi}] \leq |\mathcal{P}_{\varphi}| \cdot \frac{1}{c-1} \leq \frac{c}{c-1}.$$

We again apply McDiarmid's inequality to the random variable $Y_{\varphi}$ defined by $ Y_{\varphi}:= X_{\varphi} + c^{2/3}$. As before, $\mathbb{E}[Y_{\varphi}] = \mathbb{E}[X_{\varphi}] + c^{2/3}$.

Since the vertex $u$ appears at most once in each cone of $\Omega_{j'}$,
swapping two positions in any permutation of a half-cone in $\Omega_{j'}$
can affect $X_{\varphi}$ by at most 1. 
Thus, we can choose $d=1$ in McDiarmid's inequality.

If $Y_{\varphi} \geq s$, then $X_{\varphi} \geq s - c^{2/3}$. 
Let $P'$ be a subwalk that is counted by $X_{\varphi}$. As before, 
we can certify that $P'$ is counted by $X_{\varphi}$ by considering
only $\psi \in \Omega_{j'}$, the cone through which $P'$ passes.

To certify that $P'$ is counted by $X_{\varphi}$ we only need to reveal the two elements $\pi_{\psi}^{+}(s), \pi_{\psi}^{-}(s)$, where $s$ is the value such that one of the edges $\pi_{\psi}^{+}(s), \pi_{\psi}^{-}(s)$ contains the endpoint $u$.
Thus, $X_{\varphi} \geq s - c^{2/3}$ can be certified by the outcomes of $2(s - c^{2/3}) < 2s$ choices and we can choose $r=2$ in McDiarmid's inequality.
Thus, by a similar argument as above we obtain that $\PP[B_\varphi(j, u)] < 4e^{-c^{2/3}/64}$. Now,

\begin{align*}
&\PP\left[B_\varphi( u)\right]\le \PP\Big[\bigcup_{\forall i}B_\varphi( i,u)\Big]\le \sum_{\forall i}\PP\left[B_\varphi(i,u)\right]<4 \ell e^{-c^{2/3}/64}.
\end{align*}
Let $B_\varphi$ be the event that there exists a vertex $u$ such that more than $\ell b$ walks of $\varphi$ contain $u$. 
The number of vertices $u$ that could possibly appear in the walks $\mathcal{P}_\varphi$ is
at most $c + c^2 + ... +c^{\ell} < c^{\ell+1}$ .
%Please note that  there are no more than $c^i$ cones can be reached by going out $i$ steps from from $\varphi$, so there are no more than $c^{l+1}$ vertices can a point of some path via $\varphi$. 
Hence,
$$\PP[B_\varphi]=\PP\Big[\bigcup_{\forall u}B_\varphi(u)\Big]\le \sum_{\forall u}\PP[B_\varphi(u)]<4 c^{\ell+1} \ell e^{-c^{2/3}/64}.$$

\medskip
Let $B'_\varphi(j)$ be the event that the number of walks via $\varphi$ such that they enter 
$v_0$ at exactly their $j^{th}$-vertex is greater than $b$,
and let $B'_\varphi$ be the event that the number of walks of $\varphi$ containing $v_0$ is greater
than $\ell b$. We upper bound $\PP[B'_\varphi(j)]$. 

The argument is virtually identical to that of the estimate above.
We apply McDiarmid's inequality to the random variable $ Y_{\varphi}:= X_{\varphi} + c^{2/3}$, where $X_{\varphi}$ is the number of walks
via $\varphi$ that enter $v_0$ at the $j^{th}$ edge conditional on $\Pi$.
As before, we obtain that $\EE[X_{\varphi}] \leq c/(c-1)$, $d=1, r=2$, yielding
$\PP[ B'_\varphi(j)] \leq 4e^{-c^{2/3}/64}$.
%
%As before, we condition the partition of $\mathcal{P}_\varphi$ into arbitrary $(j-1)$-cones $\psi_1,..., \psi_r$ and give a uniform upper bound. 
%If $Y_\varphi(j)$ is the number of walks of $\varphi$ that enter $v_0$ at the $j$-vertex, then
%$$\mathbb{E}[Y_\varphi(j) \mid  \text{$\mathcal{P}_\varphi$ at level $j-1$ are partitioned in the cones $\psi_1,..., \psi_r$}] \leq 3.$$ 
%Now, applying Azuma's Inequality, with a similar argument as above we obtain that 
%$$\Pr[B'_\varphi(j)] < e^{-c^{1/3}/8} .$$ 
Thus, $\PP[B'_\varphi]< 4 \ell e^{-c^{2/3}/64}$.

\medskip

Let $J_\varphi=A_\varphi\cup B_\varphi\cup B'_\varphi$. Then
$$\PP[J_\varphi]\le \PP[A_\varphi]+\PP[B_\varphi]+\PP[ B'_\varphi]<(\ell^2+c^{\ell+1}\ell+\ell)4e^{-c^{2/3}/64}<e^{-b/100}.$$
Let $\mathcal{J}_\varphi$ be the set of events $J_\psi$ that are not
mutually independent of $J_\varphi$. 
Note that the number of permutations determining 
$J_{\varphi}$ is at most $(2c) + (2c)^2 + ... + (2c)^{\ell} < c^{\ell + 1}$.
Indeed, $c^{\ell + 1}$ is an upper bound on the number of walks of length $\ell$ that could contain an edge of $\varphi$. 
Each such permutation itself could affect at most $c + ... + c^{\ell} < c^{\ell + 1}$ events $J_{\psi}$.
Thus, $|\mathcal{J}\varphi|\le (c^{\ell+1})^2.$ 

We now apply the symmetric version of the Local Lemma. 
%Let $x_\varphi=c^{-(\ell+1)}$ for any $\varphi.$ Then,
%\begin{align*}
%x_\varphi\prod_{\psi\in \mathcal{J}\varphi}\big(1-x_\psi\big)&\ge c^{-(\ell+1)}\big(1-c^{-(\ell+1)}\big)^{|\mathcal{J}\varphi|}\\
%&>c^{-(\ell+1)}e^{-2}\\
%&>e^{-b/100}\\
%&>\Pr[J_\varphi].
%\end{align*}
To that aim, we need to have that $(c^{\ell+1})^2 e^{-b/100} < 1/4$, which clearly holds
since $\ell$ is fixed and $c$ is sufficiently large. Thus, by
Lov\'asz Local Lemma, $\PP \big[\bigcap_{\forall \varphi}\overline{J_\varphi}\big]>0$. Thus, there exists pairings of the
edges of the cones $\Gamma$ such that no event $J_\varphi$ occurs for every cone $\varphi$.

\medskip

Let $H$ be the $\ell$-path-graph obtained from $\Gamma$ by removing all bad walks and short walks.
Let $R:=G\backslash \underline{H}$. We can assume that $L$ is sufficiently large
so that $\ell^4b < 
\varepsilon (1- \varepsilon) c/2$. Then:
\begin{enumerate}
\item In every cone $\varphi$, there are no more than $\varepsilon c$ bad and short walks via it, so there are at least 
$(1-\varepsilon)c$ paths in $H$ via it. 
Hence, using the fact that $G$ is nearly equitably colored and by considering
the special $1$ and $\ell$-half cones, we obtain that for 
every $v$, there are at least $\frac{1- \varepsilon}{2\ell}d_G(v)$ paths in $H$ starting at $v$, and at least $\frac{1- \varepsilon}{2\ell}d_G(v)$ paths in $H$ ending at $v$. 
Hence, $d_H(v) \ge \frac{1-\varepsilon}{\ell}d_G(v).$ The nearly equitable $\ell$-edge-coloring implies immediately that $d_H(v) \le \frac{1+\varepsilon}{\ell}d_G(v).$

\item For every pair of vertices $u,v$, $ u \neq v$, among all walks via a cone of $u$, the ratio of walks going through $v$ is less than $\ell^2b/c< \varepsilon/2\ell$. Hence, among all walks via $u$, the ratio of walks
going through $v$ is less than $\varepsilon/2\ell$. Thus $$\frac{|\{P\in {\cal P} : u,v \in P\}|}{d_G(u)} \le \varepsilon/2\ell,$$ and, hence, ${\rm conf}(u)\le \varepsilon$.

\item In every cone, there are no more than $\ell^3b$ bad and short walks via it, so the proportion of bad walks is at most $\ell^3b/c < \varepsilon
(1-\varepsilon)/2\ell$. Hence, among all walks via a vertex $v$, the ratio of bad and short walks is less than $\varepsilon (1-\varepsilon)/2\ell$. Thus $d_R(v) < \varepsilon (1-\varepsilon) d_G(v)/2\ell$, implying $d_{R}(v)\le \varepsilon d_H(v).$\qedhere
\end{enumerate}\qedhere
\end{proof}

In the sequel, given two path-graphs $H_1$ and $H_2$ over a same graph,
we will need to grow paths of, say, $H_1$ using the paths from $H_2$.
This will essentially be achieved by considering every path $P$ of $H_1$,
incident to, say, a vertex $v$, then considering a path $P'$ incident to $v$ in $H_2$,
and just concatenating $P$ and $P'$.
So that the concatenation can be performed this way for every path of $H_1$, 
we just need $H_2$ to have enough paths,
and to make sure to evenly use these paths.
The latter requirement can be ensured by just orienting $H_2$ in a balanced way,
that is so that $|d^+(v) - d^-(v)| \leq 1$ for every vertex $v$,
and choosing, as $P'$, a path out-going from $v$.
All such out-going paths are called \textit{private paths of $v$} throughout the upcoming proofs.

The path-graph $H$ we get from $G$ after applying Theorem~\ref{dense} hence satisfies $\frac{1-\varepsilon}{\ell} d_G(v) \leq d_H(v) \leq \frac{1+\varepsilon}{\ell} d_G(v)$ for every vertex $v$.
%Because $\varepsilon$ is arbitrarily small, we can rewrite $d_H(v)/d_G(v)\in \big[ \frac{1- \varepsilon}{\ell},\frac{1+ \varepsilon}{\ell}\big]$ as $d_H(v)\approx d_G(v)/\ell$. 
If we preserve the orientation of the edges of $H$ as in the proof, 
and denote by $d_H^+(v)$ the number of paths starting from $v$ in $H$, 
we get $$\frac{1-\varepsilon}{2\ell} d_G(v) \leq d^+_H(v) \leq \frac{1+\varepsilon}{2\ell} d_G(v)$$ for every vertex $v$.
%$d_H^+(v)\approx d_G(v)/2\ell$ 
These $d_H^+(v)$ paths out-going from $v$ will hence be regarded as its private paths in what follows.

\begin{theorem}\label{ll1}
Let $\ell$ be a positive integer, and $\varepsilon'$ be a sufficiently small positive real number. 
There exists $L$ such that, for every graph $G$ with minimum degree at least $L$, 
there is an $(\ell,\ell+1)$-path-graph $H$ decomposing $G$ with
\begin{itemize}
	\item ${\rm conf}(H) \leq 1/4(\ell + 10)$, and
	\item $\frac{1-\varepsilon'}{\ell} d_G(v) \leq d_H(v) 
	\leq \frac{1+\varepsilon'}{\ell} d_G(v)$ for all vertices $v$.
\end{itemize}
\end{theorem}

\begin{proof}
Let $\varepsilon' > 0$ be sufficiently small, and set $\varepsilon:= \varepsilon' / 10$. Let 
$G_1$ be a $1/9\ell$-fraction of $G$ obtained by Proposition \ref{halfgraph}, 
and $G_2:=G\backslash G_1$. 
By applying Theorem~\ref{dense} on $G_1$ and $G_2$ with $\varepsilon$, we get two $\ell$-path-graphs $H_1$ and $H_2$  and two remainders $R_1$ and $R_2$ satisfying all properties from the statement of Theorem~\ref{dense}.
For convenience, we will keep the orientation of the edges of $H_1$ and $H_2$ given by Theorem~\ref{dense}. 
Note that $$ \tfrac{1-\varepsilon}{\ell} \cdot \
\left(\tfrac{d_G(v)}{9\ell} - 2\right)  \leq d_{H_1}(v) \leq \tfrac{1+\varepsilon}{\ell} \cdot 
\left(\tfrac{d_G(v)}{9\ell} + 2\right)$$ and 
$$ \tfrac{1-\varepsilon}{\ell} \cdot \
\left(\tfrac{(9\ell-1)d_G(v)}{9\ell} - 2\right)  \leq d_{H_2}(v) \leq \tfrac{1+\varepsilon}{\ell} \cdot 
\left(\tfrac{(9\ell-1)d_G(v)}{9\ell} + 2\right).$$
Now, we have
$\frac{1-\varepsilon}{(1+\varepsilon)(9\ell-1)} d_{H_2}(v) - 10 \leq d_{H_1}(v) \leq 
\frac{1+\varepsilon}{(1-\varepsilon)(9\ell-1)} d_{H_2}(v) + 10$
%$d_{H_1}(v)\approx d_{H_2}(v)/8$ 
for all vertices $v$.
Let $R:=R_1\cup R_2$. Then for every vertex $v$, we have
 $$d_{R}(v) = d_{R_1}(v)+d_{R_2}(v)\le\varepsilon d_{H_1}(v)+\varepsilon d_{H_2}(v) \leq 10 \varepsilon d_{H_1}(v).$$
 
%where 
%$$(1-\varepsilon) 9\varepsilon d_{H_1}(v) \leq \varepsilon\left(d_{H_1}(v) + d_{H_2}(v)\right) \leq (1+\varepsilon)9\varepsilon d_{H_1}(v).$$
 
Arbitrarily orient the edges of $R$. In our construction, every step consists in extending an arc $vu$ of $R$ using a private (i.e., outgoing) $\ell$-path starting at $v$ in $H_1$ that does not contain $u$ -- thus forming an $(\ell+1)$-path. 
Since the conflict ratio of $H_1$ satisfies
${\rm conf}(H_1) \leq \varepsilon$, at most $\varepsilon d_{H_1}(v)$ paths in $H_1$ 
with $v$ as endpoint contain $u$. Note that the number of directed $\ell$-paths in $H_1$ starting at $v$ is $d^{+}_{H_1}(v) \geq
\tfrac{1}{2} \cdot \tfrac{(1-\varepsilon)d_{G_1}(v)}{\ell}$. 
Thus, $d^{+}_{H_1}(v) - d_R(v) > \varepsilon d_{H_1}(v)$ since $L$ can be
chosen sufficiently large. 
Hence, all the $d_R(v)$ edges can be used to form $\ell+1$-paths. 

%Once the construction is achieved, graph $R$ is empty, and $H_1$ is an $(\ell,\ell+1)$-path-graph. 
We call $H_1'$ the resulting $(\ell,\ell+1)$-path-graph obtained by concatenating paths from $H_1$ and paths from $R$.
Since $d_R(v)\le 10 \varepsilon d_{H_1}(v)$ for every $v$, the degree of $v$ in $H_1'$ is as
$$d_{H_1}(v)-10\varepsilon d_{H_1}(v) \leq d_{H_1'}(v) \leq d_{H_1}(v)	+10\varepsilon d_{H_1}(v).$$
%and, thus, we have easily that
%$$\left( \frac{(1-\varepsilon)-9\varepsilon(1+\varepsilon)}{8} \right)d_{H_2}(v) \leq d_{H_1'}(v) \leq \left( \frac{(1+\varepsilon)+9\varepsilon(1+\varepsilon)}{8} \right)d_{H_2}(v).$$

%does not change much, that is we have $\frac{1-\varepsilon}{8} d_{H_2}(v) \leq d_{H_1}(v) \leq \frac{1+\varepsilon}{8} d_{H_2}(v)$.
%$d_{H_1}(v)\approx d_{H_2}(v)/8$. 

Let $H:=H_1'\cup H_2$. Then $H$ is an $(\ell,\ell+1)$-path-graph decomposing $G$, in which we have 
%$$d_{H}(v)=  d_{H_1'}(v)+ d_{H_2}(v)\approx d_{G_1}(v)/\ell+ d_{G_2}(v)/\ell=d_G(v)/\ell$$
$d_{H}(v)= d_{H_1'}(v)+ d_{H_2}(v)$ for all vertices $v$. Thus, 
$$d_{H_1}(v)	- 10\varepsilon d_{H_1}(v) + d_{H_2}(v) \leq 
d_H(v) \leq d_{H_1}(v)	+10\varepsilon d_{H_1}(v) + d_{H_2}(v).$$
Thus, $$\tfrac{1-\varepsilon}{\ell} d_{G}(v)	- 10\varepsilon d_{H_1}(v) + 1 \leq 
d_H(v) \leq \tfrac{1+\varepsilon}{\ell} d_{G}(v)	+10\varepsilon d_{H_1}(v) + 1.$$

\noindent Since $\varepsilon' = 10 \varepsilon$, we obtain
that $$\tfrac{1-\varepsilon'}{\ell} d_{G}(v)	\leq 
d_H(v) \leq \tfrac{1+\varepsilon'}{\ell} d_{G}(v).$$

\medskip

Observe also that $d_{H_1'}(v)/ d_{H_2}(v) \leq 1/6 \ell$.

\noindent Thus,

$${\rm conf}(H) \le {\rm conf}(H_2) + {\rm conf}(H_1')/ 6 \ell \leq \varepsilon + 1/ 6 \ell  < 1/4(\ell + 10),$$
as required.
\end{proof}

\section{Path-trees} \label{section:path-trees}

This part is the combinatorial core of our paper. We need here 
to show the existence of particular path-trees, namely $(\ell,2\ell)$-path-trees under mild connectivity and minimum degree requirements. 
These $(\ell,2\ell)$-path-trees will play a crucial role to insure that some path-graph has 
all of its vertices being of even degree.
However, directly getting an $(\ell,2\ell)$-path-tree seems 
a bit challenging, and we will follow a long way for this, starting 
with a $(1,2)$-path-tree and making its paths grow. 

We start off with the following lemma 
which is the key for the drop of the 
large edge-connectivity requirement. 

\begin{lemma} \label{lemma:13-tree}
Every 2-edge-connected multigraph has
a subcubic spanning $(1,2)$-path-tree.
\end{lemma}

\begin{proof} 
Let $G$ be connected and bridgeless. A \emph{structured-tree} $T$ on $G$ is 
a strongly connected digraph whose vertices are subsets $X_i$ of $V(G)$ 
satisfying the following properties:

\begin{itemize}
\item The $X_i$'s form a partition of $V(G)$.
\item The arcs of $T$ are of two types: the {\em forward
arcs} forming a rooted out-arborescence $A$, and the 
{\em backward arcs}, always directed from a vertex
to one of its ancestors in $A$.
\item Every arc $X_iX_j$ corresponds to some edge
$x_ix_j \in E(G)$ such that $x_i \in X_i$ and $x_j \in X_j$.

\item There is at most one backward arc leaving each vertex $X_i$ (unless $T$ is rooted at $X_i$).

\item Internal vertices of $A$ are singletons.

\item Every leaf $X_i$ of $A$ is spanned by a $(1,2)$-path-tree $T_i$ on $G$ with maximum
degree 3.

\item The (unique) forward and backward arcs incident to a leaf $X_i$ have 
endpoints in $T_i$ with degree at most two, and if these endpoints coincide,
the degree is at most one in $T_i$. In other words, adding the arcs as edges
of $T_i$ preserves maximum degree 3.
\item Every edge of $G$ is involved in at most one arc of $T$ and one path
of $T_i$. In other words, the edges of $G$ involved in $T$ and the $T_i$'s
are distinct.
 
\end{itemize}
 We first show that $G$ has a structured-tree $T$. For this,
fix a vertex $x$ and compute a Depth-First-Search tree $A$ from $x$.
Orient the edges of $A$ from $x$ to form the forward arcs. By the DFS 
property, every edge of $G$ not in $A$ joins vertices which 
are parents. Orient these edges from the descendent 
to the ancestor: these are our backward arcs. This is the 
classical algorithm to find a strongly connected orientation
of a bridgeless graph. Since we need to 
keep at most one backward arc issued from every vertex, 
we only keep the arc going to the lowest ancestor.
Note that we obtain a structured-tree $T$, where each $X_i$ is a 
singleton vertex in $G$ and every leaf $T_i$ is a trivial $(1,2)$-path-tree
on one vertex.

We now prove that every structured-tree $T$ with at least two vertices on $G$ can be reduced
to one with less vertices. This will imply that $T$ can be reduced to a single
vertex $X_i=V(G)$, hence providing the subcubic spanning $(1,2)$-path-tree $T_i$.

We start by deleting the backward arcs of $T$ which are not needed for strong connectivity.
Then we consider an internal vertex $X_j=\{x_j\}$ of $A$ with maximal height.
Let $X_1,X_2,\dots ,X_r$ be the (leaf) children of $X_j$. Each forward 
arc $X_jX_i$ corresponds to an edge $x_jx_i$, where $x_j\in X_j$ and $x_i\in X_i$.
Each of these leaves 
$X_i$ is the origin of a backward arc $X_iX'_i$ which 
we write $y_ix'_i$, where $y_i\in X_i$ and $x'_i\in X'_i$. We assume that 
our enumeration satisfies that $X'_{i+1}$ is always an ancestor of $X'_{i}$
(possibly equal to $X'_{i}$). We now discuss the different reductions, in which the
conditions of structured-trees are easily checked to be preserved.

\begin{itemize}
\item If $X_j$ has only one child $X_1$ and is not the origin of a 
backward arc, we merge $X_1$ and $X_j$ into a unique leaf $X_{1j}$
spanned by the $(1,2)$-path-tree $T_1\cup \{x_1x_j\}$. If $X_j$ is the root,
we are done, otherwise we let the forward arc entering
$X_{1j}$ be the one entering $X_j$, and the backward arc leaving
$X_{1j}$ be $X_{1j}X'_1$ (thus corresponding to the edge $y_1x'_1$).

\item If $X_j$ has only one child and is the origin of a 
backward arc $X_jX'_j$, we merge $X_1$ and $X_j$ into a unique leaf $X_{1j}$
spanned by the $(1,2)$-path-tree $T_1\cup \{x_1x_j\}$. The forward arc entering
$X_{1j}$ is the one entering $X_j$, and the backward arc leaving
$X_{1j}$ is the one of $X_{j}X'_j$.

\item If $X_j$ has at least three children, or $X_j$ has two children and is the origin of a 
backward arc, observe that deleting $X_1$ and $X_2$ from $T$
preserves strong connectivity. Hence we merge $X_1$ and $X_2$ into a unique leaf $X_{12}$
spanned by the $(1,2)$-path-tree $T_1\cup T_2\cup \{x_1x_jx_2\}$. The forward arc entering
$X_{12}$ is $x'_1y_1$ (hence reversing the backward arc $X_{1}X'_1$), and the backward arc leaving
$X_{12}$ is $X_{12}X'_2$ corresponding to $y_2x'_2$.

\item The last case is when $X_j$ has two children $X_1$ and $X_2$ and is not
the origin of a backward arc. Here we merge $X_1,X_2,X_j$ into a unique leaf $X_{12j}$
spanned by the $(1,2)$-path-tree $T_1\cup T_2\cup \{x_1x_j\}\cup \{x_2x_j\}$. If $X_j$ is the root,
we are done, otherwise we let the forward arc entering
$X_{12j}$ be the one entering $X_j$, and the backward arc leaving
$X_{12j}$ be $X_{12j}X'_2$ (thus corresponding to $y_2x'_2$).\qedhere
\end{itemize}
\end{proof}

We now turn our $(1,2)$-path-tree into a $(1,k)$-path-tree. For this we need 
to feed our original connected bridgeless graph $G$ (in which we find 
the subcubic $(1,2)$-path-tree) with some {\em additional graph} $H$ with sufficiently large degree.

\begin{lemma} \label{lemma:from-1k-to-1kp-tree}
Let $G=(V,E)$ be a graph. Let $T$ be a spanning $(1,k)$-path-tree of $G$, where $k \geq 2$. 
Let $H$ be some additional graph on $V$, edge-disjoint from $G$, with the property that 
$d_H(v)\geq 2(d_T(v)+2k)$ for all vertices $v$ of $G$. Then $G\cup H$ is spanned by 
a $(1,k+1)$-path-tree $T'$.
\end{lemma}

\begin{proof}
Start by arbitrarily orienting the edges of $H$ in a balanced way so that every 
vertex $v$ of $H$ has outdegree at least $d_T(v)+2k$. Every vertex is hence provided with
a set of \textit{private edges} in $H$, namely its outgoing arcs. We will use these private edges 
to transform $k$-paths of $T$ into $(k+1)$-paths.

In this proof, a \emph{structured-tree} $T'$ on $G$ is a rooted $(1,k)$-path-tree 
whose vertices are subsets $X_i$ partitioning $V(G)$ and
satisfying the following properties: 

\begin{itemize}

\item If $X_iX_j$ is an edge in $T'$, then there exists
a corresponding 1-path or $k$-path $x_ix_j \in E(T)$, where 
$x_i \in X_i$ and $x_j \in X_j$. 

\item If $X_j$ has 
children $X_1,...,X_r$ in $T'$ then there is a unique $x_j \in X_j$
such that $x_1x_j,..., x_rx_j$ are the corresponding paths in $E(T)$.
We call $x_j$ the {\em center} of  $X_j$.

%\item Internal vertices are singletons.
%\item $T$ is 2-edge-connected as a multigraph.
\item Every vertex $X_i$ of $T'$ is spanned by a $(1,k+1)$-path-tree $T'_i$. 
\end{itemize}

Initially, let $T'$ be the  structured-tree $T$, where each $X_i$ is a singleton element 
$\{x_i\}$ in $V(T)$. Note that all the vertices of $T'$ are trivial $(1,k+1)$-path-trees.
Our goal is to reduce $T'$ to one single vertex $X_i$, hence 
providing a spanning $(1,k+1)$-path-tree $T'_i$. We will always make sure that 
every center $x_j$ has at least $r+2k$ private edges, where $r$ is 
the number of children of  $X_j$.
Let us now show that $T'$ can be reduced
to one with less vertices (unless $T'$ is a single vertex).

We consider an internal vertex $X_j$ of $T'$ with maximal height.
Let $X_1,\dots ,X_r$ be the (leaf) children of $X_j$ corresponding to 
paths $x_1x_j,..., x_rx_j$, where $x_j$ is the center of $X_j$. If one of these paths, say $x_1x_j$,
is an edge, we simply create a new vertex $X_{1j}$ by concatenating $X_1$ and $X_j$
and letting $T'_{1j}=T'_{1}\cup T'_{j}\cup \{x_1x_j\}$. So we can assume that 
every $x_ix_j$-path has length $k$.
We now discuss the different reductions, in which the
conditions of structured-trees are easily checked to be preserved.

\begin{itemize}
\item If $X_j$ has at least two children $X_1$ and $X_2$, we consider 
a private neighbor $y$ of $x_j$ which is not a vertex of the paths 
$x_1x_j$ and $x_2x_j$. Such a $y$ exists since $x_j$ has at least 
$2k+2$ private neighbors. We assume that $y\in X_k$, and free to 
exchange $X_1$ and $X_2$, we can assume that $X_k\neq X_1$.
Call $P$ the $(k+1)$-path obtained by concatenating the $k$-path $x_1x_j$ 
with the edge $x_jy$.
To conclude, we add $X_1$ to the set $X_k$ to form the 
set $X_{1k}$ which is spanned by $T'_{1k}=T'_{1}\cup T'_{k}\cup \{P\}$.
Here $x_j$ loses one private edge, but $X_j$ has one child less.

\item If $X_j$ has only one child $X_1$, we again consider 
a private neighbor $y$ of $x_j$ which is not a vertex of the path 
$x_1x_j$. If $y\in X_k$ and $X_k\neq X_1$, we conclude as in 
the previous case. If $X_k=X_1$, we add $X_1$ to the set $X_j$ to form the 
set $X_{1j}$ which is spanned by the $(1,k+1)$-path-tree $T'_{1j}=T'_{1}\cup T'_{j}\cup \{x_jy\}$.
Here $x_j$ loses one private edge, but $X_j$ has no more children.\qedhere
\end{itemize}\end{proof}

The next result follows from  Lemma~\ref{lemma:13-tree} and repeated applications of Lemma~\ref{lemma:from-1k-to-1kp-tree}:

\begin{corollary} \label{corollary:existence-1-lp1-tree}
For every $\ell$, there exists $L$ such that if $G=(V,E)$ is a 
$2$-edge-connected graph and $H$ is some additional graph
on $V$ with minimum degree at least $L$, then one can form a spanning 
$(1, \ell+1)$-path-tree $T$ where $d_T(v)\leq d_H(v)$ for all
vertices $v$.
\end{corollary}

\begin{proof}
We first apply Lemma~\ref{lemma:13-tree} to get a subcubic 
$(1,2)$-path-tree $T_0$ from $G$. Fix $\varepsilon_1 >0$
to be a sufficiently small number. We choose a sequence 
of edge-disjoint subgraphs $H_1,\dots ,H_{\ell-1}$ of $H$,
where each $H_i$ is an $\varepsilon_i$-fraction of 
$H$, where $\varepsilon _{i+1} = 4\varepsilon _{i}$ for all
$i$. Free to choose $L$ large enough as a function of 
$\varepsilon_1$, we can clearly obtain the desired subgraphs 
$H_1, \dots H_{\ell - 1}$ by repeatedly applying 
Proposition \ref{halfgraph}.  Since $L$ is sufficiently large,
for each vertex $v$, we have that $d_{H_1}(v) \geq \varepsilon_1 L - 10 {\ell}^{
\ell} > 2d_{T_0}(v) + 4 \ell$. Thus, by   
Lemma~\ref{lemma:from-1k-to-1kp-tree}, we can use 
$H_1$ to extend $T_0$ into a $(1,3)$-path-tree $T_1$.
Note that $d_{T_1}(v)\leq d_{T_0}(v)+d_{H_1}(v)$ . Now we have that
$d_{H_2}(v)\geq 3.5 d_{H_1}(v) > 2 d_{T_1}(v) + 4 \ell $, and thus, we can 
again use 
$H_2$ as an additional graph to extend $T_1$ into 
a $(1,4)$-path-tree $T_2$ with $d_{T_2}(v) \leq d_{T_0}(v)+d_{H_1}(v) + d_{H_2}(v)$. We iterate this process to form our
$(1, \ell+1)$-path-tree $T$. Note that $d_T(v) \leq d_{T_0}(v) +
\sum_{i=1}^{\ell-1} d_{H_i}(v) < L \leq d_H(v)$, where the 
second to last inequality follows from the fact that we can 
choose $\varepsilon _{1}$ to be arbitrarily small.
\end{proof}

Our ultimate goal now is to find path-trees where the lengths 
of the paths are multiple of some fixed value $\ell$. One way to do so is to transform $(1,\ell +1)$-path-trees into $(\ell,2\ell)$-path-trees.
Note that if $\ell$ is even, and our graphs $G$ and $H$ are bipartite
with same bipartition, then there
is no spanning $(\ell,2\ell)$-path-tree since an even path always  
connects a partite set with itself. The next result asserts
that we can nevertheless connect each partite set separately.

\begin{lemma}\label{l2l}
For every even integer $\ell$, there exists $L$ such that if $G=(V,E)$ is a 
$2$-edge-connected bipartite graph with vertex partition $(A,B)$
and $H$ is some additional bipartite graph with vertex partition $(A,B)$
and minimum degree at least $L$, then one can form an
$(\ell, 2\ell)$-path-tree $T$ spanning $A$ where $d_T(v)\leq d_H(v)$ for every vertex $v$.
\end{lemma}

\begin{proof}
We first use a small $\varepsilon$-fraction of $H$ (and still call 
$H$ the additional graph minus this fraction for convenience) in order to apply
Corollary~\ref{corollary:existence-1-lp1-tree}. We can then 
obtain a spanning $(1,\ell +1)$-path-tree $T$ where $d_T(v)\leq \varepsilon d_H(v)$ for all
vertices $v$. Note that $\varepsilon > 0$ can be taken arbitrarily small
since we can take $L$ so that $\varepsilon L$ is sufficiently large
to apply Corollary~\ref{corollary:existence-1-lp1-tree}. 
We now apply Theorem~\ref{dense}
on $H$ to find an $(\ell-1)$-path-graph $H'$ (while
preserving the balanced orientation given by the proof) on $H$ with ${\rm conf}(H') \leq \varepsilon$ 
and $$\frac{1-\varepsilon}{\ell}d_H(v) \leq d_{H'}(v) \leq \frac{1+\varepsilon}{\ell}d_H(v)$$ for all vertices $v$.
%and $d_{H'}(v)\approx d_H(v)/\ell$ for all vertices $v$. 
%Orienting the edges of $\tilde{H'}$ in a balanced way, the outdegree of every vertex is hence roughly $d_H(v)/2\ell$,
%hence providing to every vertex $v$ this many \textit{private paths} in $H'$, namely the paths outgoing from $v$ in the orientation.
%Assuming $\varepsilon$ is small enough, this is at least $4d_T(v)$ private $(\ell-1)$-paths.

In our construction, every step consists in extending a path $P$ of $T$ 
starting at some vertex $v$ using a private (i.e. out-going) $(\ell -1)$-path from $H'$. This will form either an $\ell$-path 
or a $2\ell$-path. According to the conflict ratio assumption, every such $P$ is conflicting with at most $\varepsilon|P| < |P|/8$ private paths of $v$, 
since $\varepsilon$ can be chosen to be sufficiently small.
In our upcoming process, the total number of private 
paths of $v$ we will use is at most $d_T(v) \leq \varepsilon d_H(v)$, thus at most 
1/4 of the total number of private paths of $v$ since
$d_{H'}^{+}(v) \geq \frac{1-\varepsilon}{2 \ell}d_H(v)$. Hence, even if we have already used 1/4 of the 
private paths of $v$, and we need a private path of $v$
which is non-conflicting with two paths of $T$ incident 
to $v$, we can still find one. Thus, in the upcoming arguments, 
we always assume that a private path is available 
whenever we need one.

We now turn to the construction of the 
$(\ell, 2\ell)$-path-tree $T'$ spanning $A$.
A \emph{structured-tree} $T'$ on $G$ is a rooted tree in which the vertices are 
disjoint subsets $X_i$ whose union covers a subset of $V(G)$ containing $A$
with the following properties:

\begin{itemize}

\item If $X_iX_j$ is an edge in $T'$, then there exists
a corresponding 1-path or $(\ell+1)$-path $x_ix_j \in E(T)$, where 
$x_i \in X_i$ and $x_j \in X_j$. 

\item If $X_j$ has 
children $X_1,...,X_r$ then there is a unique $x_j \in X_j$
such that $x_1x_j,..., x_rx_j$ are the corresponding paths in $T$.
We call $x_j$ the {\em center} of  $X_j$.

\item Every vertex $X_i$ containing an element of $B$ is a singleton, i.e. 
$X_i = \{x_i\}$.  

\item Every vertex $X_i$ of $T'$ is spanned by an $(\ell, 2\ell)$-path-tree $T'_i$. 

\end{itemize}

We again start with $T'$ equal to $T$ in the sense that all
$X_i$'s are singletons, and all $T'_i$'s are trivial $(\ell, 2\ell)$-path-trees.
We root $T'$ at some arbitrary vertex of $A$. Again
our goal is to show that we can reduce $T'$ until it is 
reduced to its root, which will therefore be equal to the set 
$A$, covered by an $(\ell, 2\ell)$-path-tree. Note that since
$\ell$ is even, we always have that an edge $X_iX_j$
of $T'$ connects a vertex of $B$ and a subset of $A$.

Observe first that if $T'$ has a leaf in $B$, we can simply delete 
it and keep our properties. We can then assume that all leaves are 
subsets of $A$.
We consider an internal vertex $X_j$ of $T'$ with maximal height.
Let $X_1,X_2,\dots ,X_r$ be the (leaf) children of $X_j$ corresponding to 
the paths $x_1x_j,..., x_rx_j$. Note that all $X_i$'s are subsets of 
$A$, and that $X_j=\{x_j\}$ is in $B$.
We now discuss the different reductions, in which the
conditions of structured-trees are easily checked to be preserved.

\begin{itemize}
\item If $X_j$ has at least two children $X_1$ and $X_2$, we consider 
a private $(\ell-1)$-path $x_jy$ which is non conflicting with the paths 
$x_1x_j$ and $x_2x_j$. Note that by parity, $y$ belongs to $A$. Let 
$y\in X_k$, and free to 
exchange $X_1$ and $X_2$, we can assume that $X_k\neq X_1$.
We denote by $P'$ the path obtained by concatenating the path $x_1x_j$ 
with $x_jy$. Note that 
$P'$ is an $\ell$-path or a $2\ell$-path.
To conclude, we add $X_1$ to the set $X_k$ to form the 
set $X_{1k}$ which is spanned by $T'_{1k}=T'_{1}\cup T'_{k}\cup \{P'\}$.
Note that $x_j$ loses a private path, but $X_j$ has one child less.

\item If $X_j$ has only one child $X_1$, we consider the parent $X_k$
of $X_j$ in $T'$. Note that $X_k$ is a subset of $A$. We denote by 
$x_jx_k$ the path of $T$ joining $X_j$ and $X_k$. We consider 
a private $(\ell-1)$-path $x_jy$ which is non conflicting with the paths 
$x_1x_j$ and $x_kx_j$. If $y\notin X_1$, we conclude as in 
the previous case. If $y\in X_1$, we add $X_1$ to the set $X_k$ to form the 
set $X_{1k}$ which is spanned by the $(\ell, 2\ell)$-path-tree $T'_{1k}=T'_{1}\cup T'_{k}\cup \{x_kx_jP[x_jy]\}$,
where $P[x_jy]$ is the subpath of $P$ from $x_j$ to $y$.
Here $x_j$ loses one private edge, but $X_j$ has no more children (and can then be deleted since it becomes a leaf).\qedhere
\end{itemize}
\end{proof}

We will also need the following lemma.

\begin{lemma}\label{llp1}
Let $\ell$ be a positive integer. There exists 
$L$ such that if $G_1=(V,E)$ is a $2$-edge-connected graph and $G_2=(V,F)$ is a 
graph of minimum degree at least $L$ edge-disjoint from $G_1$, 
then there is a connected $[\ell, \ell +3]$-path-graph $H$ decomposing $G_1 \cup G_2$ with ${\rm conf}(H)<
\tfrac{1}{2(\ell + 10)}$.
\end{lemma}

\begin{proof} 
Start by applying Lemma~\ref{lemma:13-tree} to get a subcubic 
$(1,2)$-path-tree $T$ spanning $G_1$, and put the non-used edges of $G_1$ in $G_2$.
Still calling this graph $G_2$, we decompose $G_2$ into a $1/(5\ell)$-fraction $R_1$
and a $1-1/(5 \ell)$-fraction $R_2$, by Proposition \ref{werra}. Thus, by Theorem~\ref{ll1},
$G_2$ can then be decomposed
into two $(\ell,\ell+1)$-path-graphs $H_1$ and
$H_2$, respectively, both having conflict ratio at most 
$\tfrac{1}{4(\ell+10)}$, 
and verifying $$\frac{1-\varepsilon}{(5\ell-1)(1+\varepsilon)} d_{H_2}(v) \leq d_{H_1}(v) 
\leq \frac{1+\varepsilon}{(5\ell-1)(1-\varepsilon)} d_{H_2}(v)$$ for all vertices $v$, for any $\varepsilon$.

In our construction, every step consists in extending a path $P$ of $T$ 
starting at $v$ using a private $(\ge \ell)$-path starting at $v$ in $H_1$
(where we recall that the private paths at any vertex are 
its out-going paths in a balanced orientation of $H_1$).
This will form a $(\ge \ell)$-path. By the assumption
on the conflict ratio, every $P$ is conflicting with at most, say, half of the private paths of $v$.
Because $T$ is subcubic, the total number of private 
paths of $v$ we will need is at most six. Since $L$ can be chosen
so that $\tfrac{L}{5} \cdot \left(1-\tfrac{1}{4 (\ell + 10)} \right)$  
is arbitrarily large, we can hence assume we have enough private paths for the whole process.

We now turn to the construction of the spanning 
$(\ge \ell)$-path-tree $T'$ from $T$ and $H_1$.
A \emph{structured-tree} $T'$ on $V$ is a rooted tree in which the vertices are 
disjoint subsets $X_i$ partitioning $V$
with the following properties:

\begin{itemize}

\item If $X_iX_j$ is an edge in $T'$, then there exists
a corresponding 1-path or $2$-path $x_ix_j \in E(T)$, where 
$x_i \in X_i$ and $x_j \in X_j$. 

\item Every vertex $X_i$ of $T'$ is spanned by a $(\ge \ell)$-path-tree $T'_i$. 

\end{itemize}

We again start with $T'$ being equal to $T$ in the sense that all
$X_i$'s are singletons, and all $T'_i$'s are trivial $(\ge \ell)$-path-trees.
We root $T'$ at some arbitrary vertex. Again
our goal is to show that we can reduce $T'$ until it is 
reduced to its root, which will therefore be a spanning $(\ge \ell)$-path-tree. 

We consider a leaf $X_1$ of $T'$ with direct ancestor $X_j$.
Then there exists a path $x_1x_j$ of $T'$ having length 1 or 2. We pick a private path $x_jy\in H_1$ not conflicting with the path
$x_1x_j$. Assume $y\in X_k$. If $X_k\neq X_1$,
we denote by $P$ the path obtained by concatenating $x_1x_j$ and $x_jy$.
Then we add $X_1$ to $X_k$ to form the 
set $X_{1k}$ being spanned by $T'_{1k}=T'_{1}\cup T'_{k}\cup \{P\}$. 
If $X_k= X_1$, we add $X_1$ to $X_j$ to form the set $X_{1j}$ being spanned by $T'_{1j}=T'_{1}\cup T'_{j}\cup \{x_jy\}$. 
We choose a private path $x_jz$ in $H_1$ being not conflicting with $x_1x_j$, and 
concatenate these two paths to get a path $x_1z$ that we put back into $H_1$. 

Once the procedure above is finished, we end up with a 
spanning $(\ge \ell)$-path-tree $T'$ and an $(\ell, \ell + 1)$-path-graph $H'_1$, where
$H'_1$ is the path-graph remaining from $H_1$ after we have used some of its paths to obtain $T'$. Let 
$H:=T'\cup H'_1\cup H_2$. Then $H$ covers all edges of $G$. Note also that $H$ is an $[\ell, \ell + 3]$-path-graph.
% since no path of $T'$ or $H'_1$ has length greater than $\ell + 3$.
Since $d_{T'\cup H'_1}(v) \le d_{H_1}(v)+3$ for every vertex $v$ and we can choose
$\epsilon$ to be sufficiently small, we 
have $d_{T'\cup H'_1}(v) \le d_{H_2}(v)/4 (\ell+10)$ for every vertex $v$. Thus, 
$${\rm conf}(H) \leq {\rm conf}(H_2) + \frac{{\rm conf}(T'\cup H'_1)}{4(\ell+10)} < \frac{1}{4(\ell+10)} 
+ \frac{1}{4(\ell+10)} \leq \frac{1}{2(\ell + 10)}.$$
\end{proof}

%%%%%%%%%%%%%%%%%%%%%%%%%%%%%%%%%%%%%%%%%%%%%%%%%%%%%%%%%%%%%%%%%%%%%
%%%%%%%%%%%%%%%%%%%%%%%%%%%%%%%%%%%%%%%%%%%%%%%%%%%%%%%%%%%%%%%%%%%%%
%%%%%%%%%%%%%%%%%%%%%%%%%%%%%%%%%%%%%%%%%%%%%%%%%%%%%%%%%%%%%%%%%%%%%
%%%%%%%%%%%%%%%%%%%%%%%%%%%%%%%%%%%%%%%%%%%%%%%%%%%%%%%%%%%%%%%%%%%%%
%%%%%%%%%%%%%%%%%%%%%%%%%%%%%%%%%%%%%%%%%%%%%%%%%%%%%%%%%%%%%%%%%%%%%
%%%%%%%%%%%%%%%%%%%%%%%%%%%%%%%%%%%%%%%%%%%%%%%%%%%%%%%%%%%%%%%%%%%%%
%%%%%%%%%%%%%%%%%%%%%%%%%%%%%%%%%%%%%%%%%%%%%%%%%%%%%%%%%%%%%%%%%%%%%
%%%%%%%%%%%%%%%%%%%%%%%%%%%%%%%%%%%%%%%%%%%%%%%%%%%%%%%%%%%%%%%%%%%%%
%%%%%%%%%%%%%%%%%%%%%%%%%%%%%%%%%%%%%%%%%%%%%%%%%%%%%%%%%%%%%%%%%%%%%
%%%%%%%%%%%%%%%%%%%%%%%%%%%%%%%%%%%%%%%%%%%%%%%%%%%%%%%%%%%%%%%%%%%%%
%%%%%%%%%%%%%%%%%%%%%%%%%%%%%%%%%%%%%%%%%%%%%%%%%%%%%%%%%%%%%%%%%%%%%
%%%%%%%%%%%%%%%%%%%%%%%%%%%%%%%%%%%%%%%%%%%%%%%%%%%%%%%%%%%%%%%%%%%%%
%%%%%%%%%%%%%%%%%%%%%%%%%%%%%%%%%%%%%%%%%%%%%%%%%%%%%%%%%%%%%%%%%%%%%
%%%%%%%%%%%%%%%%%%%%%%%%%%%%%%%%%%%%%%%%%%%%%%%%%%%%%%%%%%%%%%%%%%%%%
%%%%%%%%%%%%%%%%%%%%%%%%%%%%%%%%%%%%%%%%%%%%%%%%%%%%%%%%%%%%%%%%%%%%%
%%%%%%%%%%%%%%%%%%%%%%%%%%%%%%%%%%%%%%%%%%%%%%%%%%%%%%%%%%%%%%%%%%%%%

\section{Edge-partitioning a graph into $\ell$-paths} \label{section:main}

We now have all ingredients to prove our main results, i.e. Theorems~\ref{theorem:24} and~\ref{theorem:4eulerian2}.
We start off with the proof of Theorem~\ref{theorem:24}.

\begin{proof}[Proof of Theorem~\ref{theorem:24}]
Without loss of generality, we assume that $\ell$ is even (the statement 
for $2k$ implying the statement for $k$).
First of all, we consider a maximum cut $(V_1,V_2)$ of $G$, and just 
keep the set of edges $F$ across the cut. We call $G'$ the graph $(V,F)$.
Observe that $G'$ is at least $12$-edge-connected and has minimum
degree at least $d_\ell/2$.

By Proposition \ref{prop: arc-strong}, there is an orientation $D$ of $G'$
such that $D$ is $6$-arc-strong and with $d^{+}(v)$ and $d^{-}(v)$ differing by at most one for every vertex $v$.
By applying Proposition \ref{prop: disj-arbor} to $D$ with some vertex $z$, 
we obtain 6 arc-disjoint out-arborescences, $T_1, \dots ,T_6$, rooted at $z$.  
Since each vertex $v$ has in-degree at most~$1$ in $T_i$ ($z$ has in-degree 0), and $d^{+}_D(v)$ and $d^{-}_D(v)$
differ by at most one, the graph $T_1 \cup ... \cup T_6$ is $1/2$-sparse in $G'$.

Call now $G_1:=T_1\cup T_2$, $G_2:=T_3 \cup T_4$, $G_3:=T_5 \cup T_6$,
and let $R$ be the graph consisting of all the edges of $F$ which are not 
in $G_1, G_2, G_3$. Observe that $G_1, G_2, G_3$ are connected
and bridgeless. Furthermore, the graph $G_1 \cup G_2 \cup G_3$ is $1/2$-sparse in $G'$, and hence $R$ is $1/2$-dense in $G'$.

We turn $G_1$ into an $(\ell,2\ell)$-path-tree as follows: we consider
a small $\varepsilon$-fraction $R_1$ of $R$, and apply Lemma~\ref{l2l} (with 
$G_1$ for $G$ and $R_1$ for $H$) to form an
$(\ell,2\ell)$-path-tree $T'$ spanning $V_1$ in which $d_{T'}(v)\leq d_{R_1}(v)$ for all
vertices $v \in V_1$. 
In other words, $T'$ is $\varepsilon$-sparse in $R$
(here we still call $R$ the original one minus $R_1$). Similarly, we can obtain, from $G_2$, a $\varepsilon$-sparse 
$(\ell,2\ell)$-path-tree $T''$ spanning $V_2$.
We still consider (neglecting 
the two $\varepsilon$-fractions) that $R$ is $1/2$-dense in $G'$. Add all edges of $E \backslash F$ to $R$.

Now, $G=G_3\cup\underline{T'}\cup\underline{T''}\cup R$. We claim that we can 
remove a collection of $\ell$-paths or $2\ell$-paths
from the path-tree $T'$ spanning $V_1$ in a way so that we can obtain that 
at most one vertex of $V_1$ has odd degree in $G$. Indeed, note that if $T$
is a tree and $X$ is an even subset of $V(T)$, then there exist a 
set of edges $F \subseteq E(T)$ such that for each vertex $x$, $d_F(x)$ is odd if and only if
$x \in X$ (one way to see this is to note that the characteristic vector of $X$ is in the span 
of the incidence matrix of $T$). In particular, denoting by $X_1$ the set of all odd degree vertices of 
$G_3 \cup T'' \cup R$ inside $V_1$ (and possibly removing one vertex of $X_1$ to make $X_1$ of even size)
we can find a subtree $F'$ of $T'$ such that $d_{F'}(v)$ is odd if and only if $v \in X_1$. In other words,
removing the $\ell$ or $2 \ell$-paths of $T'$ corresponding to $F'$ leaves $G$ with every vertex of $V_1$
(except possibly one) with even degree.
Similarly, we remove paths
of the path-tree $T''$ spanning $V_2$ so that 
at most one vertex of $V_2$ has odd degree. 

We still call $G$ the remaining graph after the procedure, and we add the remaining edges of $\underline{T'}\cup\underline{T''}$ to $R$. Then $G=G_3\cup R$. 
Note that $G_3$ is 2-edge-connected, and $R$ is 1/4-dense in $G$.
By applying Lemma~\ref{llp1} (with 
$G_3$ for $G_1$ and $R$ for $G_2$), we can express $G$ as a connected
$[\ell, \ell +3]$-path-graph $H$ with ${\rm conf}(H)< 1/2(\ell+10)$. Note that $d_G(v)-d_H(v)$ is 
even for every vertex $v$ -- so the degree of every vertex in $H$ is even, 
except (possibly) for two vertices $v_1\in V_1$ and $v_2\in V_2$. In this case, 
we add a dummy $\ell$-path from $v_1$ to $v_2$ in $H$ to make $H$ eulerian. By 
Theorem \ref{eulerian}, we get that $H$ has a non-conflicting eulerian tour 
from which we can deduce the desired decomposition.
\end{proof}

One important fact in the proof of Theorem~\ref{theorem:24} is that, when constructed, the path-graph
$H$ covers all edges of $G$. For this reason, it should be clear that the parity of the degree of every vertex
is preserved from $G$ to $H$. This simple remark implies the following interesting counterpart result on
eulerian graphs that are sufficiently edge-connected and have large enough minimum degree.

\begin{theorem} \label{theorem:4eulerian}
For every integer $\ell$, there exists $d_\ell$ such that every $4$-edge-connected eulerian graph
$G$ with minimum degree at least~$d_\ell$ has an eulerian tour with no cycle of length at most $\ell$.
\end{theorem}

\begin{proof}
Following the arguments in the second paragraph of Theorem \ref{theorem:24}, 
we can extract from $G$ two trees $T_1$ and $T_2$ so that $T_1\cup T_2$ is $1/2$-sparse 
in $G$. Let $G_1:=T_1\cup T_2$, and $G_2:=G\backslash G_1$. Then $G_1$ is 2-edge-connected, 
and $G_2$ is $1/2$-dense. Applying Lemma~\ref{llp1}, we can express $G$ as a 
connected $[\ell, \ell + 3]$-path-graph $H$ with ${\rm conf}(H) < 1/2(\ell + 10)$. Since $G$ is
eulerian, so is $H$. Hence $H$ has non-conflicting eulerian tours according to
Theorem \ref{eulerian}, and these tours do not have cycles of length at most~$\ell$ since
all paths of $H$ have length at least $\ell$.
\end{proof}

Theorem~\ref{theorem:4eulerian} now directly implies Theorem~\ref{theorem:4eulerian2}.

\end{document}